\documentclass[a4paper,oneside]{article}
\usepackage{amssymb}
\usepackage{amsmath}
\usepackage{amsthm}
\usepackage{float}
\usepackage[T1]{fontenc}
\usepackage[utf8]{inputenc}
\usepackage[british]{babel}
\usepackage{geometry}
\usepackage{caption}
\usepackage{graphicx}

\usepackage{nomencl}
\makeglossary
\makenomenclature

\newcommand{\Hyp}{\mathbb{H}}
\newcommand{\C}{\mathcal{C}}
\newcommand{\U}{\mathcal{U}}
\newcommand{\R}{\mathbb{R}}

\newcommand{\SU}{\mathbb{S}^1}
\newcommand{\N}{\mathbb{N}}
\newcommand{\Z}{\mathbb{Z}}
\newcommand{\F}{\mathcal{F}}
\newcommand{\Pants}{\mathcal{P}}

\newcommand{\G}{\Gamma}
\newcommand{\g}{\gamma}
\newcommand{\la}{\lambda}

\newcommand{\de}{\delta}
\newcommand{\tv}{\rightarrow}

\newcommand{\CC}{\mathcal{CC}}

\newcommand{\gs}{\mathcal{I}}
\newcommand{\GSt}{\mathcal{G}\widetilde{S}}
\newcommand{\Lam}{\mathcal{L}}
\newcommand{\MLam}{\mathcal{ML}}
\newcommand{\PMLam}{\mathcal{PML}}
\newcommand{\GC}{\mathcal{GC}}
\newcommand{\RP}{\mathbb{RP}}
\newcommand{\PGC}{\mathcal{PGC}}
\newcommand{\EQ}[3]{\mathcal{E}_{#1}^{#2}{(#3)} }

\newcommand{\Produit}{\Hyp^2 \times \Hyp^2}

\newtheorem{theorem}{Theorem}[section]

\newtheorem{lemme}[theorem]{Lemma}

\newtheorem{proposition}[theorem]{Proposition}

\newtheorem{corollaire}[theorem]{Corollary}

\newtheorem{definition}[theorem]{Definition}

\newtheorem*{TheoremNoCount}{Theorem}
\newtheorem*{PropNoCount}{Proposition}
\newtheorem*{Example}{Examples}

\DeclareMathOperator{\Teich}{Teich}
\DeclareMathOperator{\Isom}{Isom}
\DeclareMathOperator{\Card}{Card}
\DeclareMathOperator{\PSL}{PSL}

\DeclareMathOperator{\dil}{dil}

\DeclareMathOperator{\SL}{SL}

\DeclareMathOperator{\HDim}{Hdim}

\DeclareMathOperator{\AdS}{AdS}

\begin{document}
 
\title{ Counting closed geodesics in globally hyperbolic maximal compact AdS 3-manifolds.}
\author{Olivier Glorieux}

\maketitle

\begin{abstract}							
We propose a definition for the length of closed geodesics in  a globally hyperbolic maximal compact (GHMC) Anti-De Sitter manifold. We then prove that the number of closed geodesics of length less than $R$ grows exponentially fast with $R$ and the exponential growth rate is related to the critical exponent associated to the two hyperbolic surfaces coming from Mess parametrization. We get an equivalent of three results for quasi-Fuchsian manifolds in the GHMC setting : R. Bowen's rigidity theorem  of critical exponent, A. Sanders' isolation theorem and C. McMullen's examples lightening the behaviour of this exponent when the surfaces range over Teichmüller space. 
\end{abstract}


\section{Introduction}
A classical problem in Riemannian geometry is to count the number of closed geodesics on a manifold, or estimate its growth.  For compact negatively curved manifolds, this number grows exponentially fast and we know a very precise estimate since G. Margulis' thesis, \cite{margulis1969applications}, showing the relation with volume entropy. This exponential growth rate is called critical exponent, this is also the abscissa of convergence of the Poincaré series of the fundamental group of the manifold acting on the universal cover. 

For a wide class of manifolds we understand quite well this invariant. For hyperbolic compact manifolds $M$, it is constant equal to $\dim(M)-1$.  More generally for hyperbolic convex-cocompact manifolds it is equal to the Hausdorff dimension of the limit set on the sphere at infinity, see \cite{BishopJones} for an even more  general  result. A vast category of convex-cocompact 3-manifolds are given by quasi-Fuchsian manifolds. These are hyperbolic 3-manifolds  whose limit set of their fundamental group on the sphere at infinity is a topological circle. They are topologically the product of a surface of genus greater than 2 and $\R$. The geometry of a quasi-Fuchsian is encoded by two points in the Teichmüller space of $S$, $\Teich(S)$, through the so called Bers simultaneous uniformization \cite{bers1972uniformization}.  

The behaviour of the critical exponent for quasi-Fuchsian manifolds has been deeply studied. 
A theorem of R. Bowen \cite{bowen1979hausdorff} says that the critical exponent is greater or equal to $1$, and equality occurs if and only if the two points are the same, which geometrically says that the limit set is a round circle  of the sphere at infinity.  Quantitative theorems have also been shown. A. Sanders in \cite{sanders2014entropy} proved an isolation theorem: he showed that the critical exponent of a sequence of  quasi-Fuchsian manifolds encoded, through Bers simultaneous uniformisation, by two sequences of point goes to 1 if and only if the two points tend to the same limit in the Teichmüller space. Finally C. McMullen in \cite{mcmullen1999hausdorff} studied the behaviour of the critical exponent for sequences of quasi-Fuchsian manifolds parametrized  by a fixed surface and a sequence going to the boundary of the Teichmüller space. 

 For Lorentzian manifolds, the same counting problem makes no sense in general, since the length of a curve is not necessarily well defined.  However, there is a subclass of Lorentzian manifolds called \emph{Globally Hyperbolic Anti-de Sitter manifolds}, which are the analogous to quasi-Fuchsian manifolds, for which we propose in this article a natural definition for critical exponent. The aim of this article is to define, study and prove the counterpart of the three theorems given in the last paragraph, for this critical exponent. 

We will give a quick review of the Lorentzian manifolds we are interested in. It will turn out that their geometry is also encoded by two points in the Teichmüller space and a large part of this article will study the action of two Fuchsian representations on the product of two hyperbolic spaces: $\Hyp^2\times \Hyp^2$.  A preprint with D. Monclair and the author, \cite{glorieux2016}, continues the investigation of these invariants in a more Lorentzian perspective. 

We will finally recall the usual definition  of critical exponent and explain how it is related to other classical invariants.
\subsection{Globally hyperbolic  maximal compact anti-de Sitter manifolds}

A lot of work has been done on \emph{Globally Hyperbolic} (GH)  \emph{Anti-de Sitter} ($\AdS$) manifolds during the last decades, based on the pioneering work of G. Mess \cite{mess2007lorentz} describing the geometry of such manifolds. A concise and complete presentation of the geometric background for GH \, $\AdS$ manifolds can be found in \cite{diallo2014metriques}, and I will follow this text.  More detailed ones are \cite{barbot2007constant} and \cite{benedetti2009canonical}. 

Recall that on Lorentzian manifold there are three types of tangent vectors classified by the sign of the quadratic form. Let $Q$ denotes the quadratic form  coming from the Lorentzian scalar product and $v$ a tangent vector. We say that $v$ is spacelike if $Q(v)>0$, lightlike if $Q(v)=0$ and timelike if $Q(v)<0$. This vocabulary extends to $C^1$ curves : we say that a $C^1$ curve is spacelike (respectively lightlike, timelike), if all its tangent vectors are spacelike  (respectively lightlike, timelike).

A \emph{Cauchy surface} in a Lorentzian 3-manifold is a spacelike surface which intersects every inextendable timelike and lightlike curves exactly once.   If a  Lorentzian manifold contains a Cauchy surface it is said  \emph{globally hyperbolic}. Moreover it is said \emph{maximal} if there is no isometric embedding in a strictly larger space time, sending a Cauchy surface on a Cauchy surface. 

It follows from global hyperbolicity that every Cauchy surfaces are homeomorphic and we will suppose in this paper that Cauchy surfaces are compact surfaces of genus $g\geq 2$. The manifolds satisfying those last two conditions are called \emph{globally hyperbolic maximal compact} (GHMC).  

The Anti-de Sitter space  is a maximal symmetric Lorentzian space of constant curvature $-1$, this is the equivalent of the hyperbolic space in the Lorentzian setting. Let us present the following linear model for $\AdS$. Consider $M_2(\R)$ the space of 2 by 2 real matrices endowed with the scalar product $\eta$ induced by the quadratic form $-\det$. The signature of $-\det$ is (2,2). The level $-1$ of this quadratic form is $\SL_2(\R)$, and the restriction of $\eta$ is of signature $(2,1)$, hence it is a Lorentzian manifold called \emph{Anti-de Sitter space}. It is a Lorentzian space time of constant negative curvature $-1$.  We will consider 3-manifolds locally modeled on $\AdS$ 3-space, that is  a manifold admitting a $(\Isom_0(\AdS_3), \AdS_3)-$structure, where $\Isom_0(\AdS_3)$ is the connected component of $(Id, Id)$ in the group of isometries of $\AdS_3$.  We can see that $\Isom_0(\AdS_3)$ is isomorphic to $\SL_2(\R)\times \SL_2(\R)/(-Id, -Id)$, \cite{barbot2007constant}, where an element $(\g_L, \g_R)\in \SL_2(\R)\times \SL_2(\R)/(-Id, -Id)$ acts by left and right multiplication.  This action of $\SL_2(\R)\times \SL_2(\R) $ on $M_2(\R)$ preserves the quadratic form and therefore $-\det$ induces a bi-invariant Lorentzian metric on $\SL_2(\R)$. 

It will sometimes be more convenient for us to use projective model : $\eta$ induces on $\PSL_2(\R)$ a Lorentzian structure for which the isometry group is $\PSL_2(\R) \times \PSL_2(\R)$. Hence the holonomy of a $\AdS$ 3-manifolds, $M$,  is naturally given by two representations of $\pi_1(M)$ in $\PSL_2(\R)$. We will call these two representations \emph{left} and \emph{right} : $(\rho_L,\rho_R)$.

Let $M$ be a GHMC, $\AdS$ manifold. Recall that we supposed the Cauchy surface  to be a compact surface $S$ of genus greater than 2. Then from global hyperbolicity  we have  $\pi_1(M)=\pi_1(S)$.   G. Mess, showed  in \cite{mess2007lorentz} that for GHMC, the two  representations $(\rho_L,\rho_R)$ are faithful and discrete, in other words, they are points in the Teichmüller space of $S$, $\Teich(S)$. Conversely, he explained how to construct a GHMC manifold with these two points as left and right holonomies. His construction gives a  parametrization of  GHMC manifolds by the product of two  Teichmüller spaces of $S$. 

This can be considered as the Bers simultaneous unifomization of quasi-Fuchsian manifolds, and GHMC manifolds can be seen as the Lorentzian counterpart of quasi-Fuchsian manifolds in many aspects. Both are homeomorphic to a compact surface time $\R$. Their holonomy is given by product of Teichmüller representations. The holonomy of a GHMC manifolds gives a limit set on the boundary of $\AdS$, this limit set is a curve which is a round circle if and only if the left and right representations are the same in $\Teich(S)$. As for quasi-Fuchsian manifolds, GHMC ones have a convex core, for which the boundaries components are two hyperbolic pleated surfaces. There are open questions on the geometry of those boundaries :  we know in both cases that we can prescribe the metric on them (from the work of Epstein and Marden, \cite{marden111convex}, for quasi-Fuchsian case, and from the work of Diallo, \cite{diallo2014metriques}, for  GHMC case) and it is still an open question to know whether or not the metrics on the boundaries determine the entire manifold (these questions have been raised by W. Thurston for quasi-Fuchsian and G. Mess for GHMC).

\subsection{Dynamical invariants for quasi-Fuchsian manifolds}
There are a lot of dynamical invariants associated to quasi-Fuchsian manifolds, for example the Hausdorff dimension of the limit set, the volume entropy of the convex core, the critical exponent, the growth rate of the number of closed geodesics... It is a standard fact these invariants are equal in the context of quasi-Fuchsian manifolds, see for example the introduction of  \cite{mcmullen1999hausdorff}. A natural question is to know if those invariants can be defined in the GHMC setting and if they satisfy nice properties ?  

Let us recall the known results in the quasi-Fuchsian setting we want to obtain in the GHMC setting. 
We call $\delta_{QF}(S_0,S_1)$ the critical exponent associated to the quasi-Fuchsian manifold parametrized by $S_0$ and $S_1$ through Bers simultaneous uniformization. We would at least expect an invariant which distinguishes the Fuchsian case, ie. when $\rho_L$ is conjugated to $ \rho_R$. In the quasi-Fuchsian settings, it is given by the following rigidity theorem due to R. Bowen

\begin{theorem}\cite{bowen1979hausdorff}
$$\delta_{QF}(S_0,S_1)\geq 1.$$ 
Moreover the equality occurs if and only if $S_0=S_1$. 
\end{theorem}
It is moreover known that $\delta_{QF} <2$, for any quasi-Fuchsian manifolds. \\

In fact, there is a quantitative result of this last theorem given by A. Sanders. 
Let us recall before stating this theorem what we mean by the thick part of the set of  quasi-Fuchsian manifolds. We say that a sequence of  quasi-Fuchsian manifolds stays in the thick part of the set of Quasi-Fuchsian manifolds if the injectivity radius is bounded below by $\epsilon$, for some $\epsilon>0$.  Recently, A. Sanders showed that if a quasi-Fuchsian manifolds $M$ stays in the thick part of the Teichmüller space of $S$, the value $1$  for the   critical exponent is isolated around Fuchsian locus : 
\begin{theorem}\cite[Theorem 5.5]{sanders2014entropy}
Fix $\epsilon_0>0$, and suppose $M$ is a quasi-Fuchsian manifold in the $\epsilon_0$ thick part, parametrized by $S_0$ and $S_1$. For every $\epsilon$ there exists $\eta(\epsilon,\epsilon_0)$ such that if 
$$\delta_{QF} (S_0,S_1) \leq 1+\eta,$$
then $d_{Teich}(S_0,S_1) \leq  \epsilon,$ where $d_{Teich}$ is the Teichmüller distance on $\Teich(S)$.
\end{theorem}
We can rephrase it with sequences : 
\begin{TheoremNoCount}\cite[Theorem 5.5]{sanders2014entropy}
Let  $(M_n)$ be a sequence of  quasi-Fuchsian manifolds in the $\epsilon_0$ thick part, parametrized by $(S_n)$ and $(S_n')$, then:
$$\lim_{n\tv \infty}\delta_{QF} (S_n,S_n')=1, \quad \text{iff} \quad \lim_{n\tv \infty} d_{Teich}(S_n,S_n') =0,$$
\end{TheoremNoCount}

Finally we also expect a good behaviour at infinity and C. McMullen \cite{mcmullen1999hausdorff} gave examples of quasi-Fuchsian sequences for which we know the behaviour of critical exponent.  
\begin{theorem}\cite{mcmullen1999hausdorff}
Let $S_0$ be a fixed point on $\Teich(S)$. 
\begin{itemize}
\item Let $A$ be  a pseudo-Anosov diffeomorphism $S$ and call $S_n := A^{n} S_0$. Then 
$$\lim_{n\tv \infty}\delta_{QF}(S_0,S_n) = 2$$ 
\item Let $S_n$ be the surface obtained after pinching a disjoint set of simple closed curves on $S_0$. Then 
$$\lim_{n\tv \infty} \delta_{QF} (S_0,S_n) = \alpha <2 $$
\item Let $c$ be a simple closed curve, and $\tau$ the Dehn twist along $c$. Call $S_n := \tau^{n} S_0$. Then 
$$\lim_{n\tv \infty}\delta_{QF}(S_0,S_n) = \alpha.$$ 
\item Let $S_t$ be a surfaces path obtained by Fenchel twist along a simple closed geodesic $c$. Then there exists a $\ell_{S_0}(c)$-periodic  function $\delta $ such that: 
$$\lim_{t\tv \infty} | \delta_{QF}(S_0,S_t) -\delta(t) | =0.$$
\end{itemize}
\end{theorem}
The purpose of this paper is to give analogous statements in the GHMC setting.

Let us say a few words on the others invariants. There are several definitions  for critical exponent which all agree for quasi-Fuchsian manifolds: 
\begin{eqnarray*}
\delta_{QF} &=& \limsup_{R\tv\infty} \frac{1}{R}\log\Card\{\g \in \G \, |\, d(\g o,o ) \leq R\} \\
					&=& \limsup_{R\tv\infty} \frac{1}{R}\log\Card\{c \in \C \, |\, \ell_{QF} (c) \leq R\} \\
					&=&  \HDim(\Lambda).
\end{eqnarray*}

 Generally, the critical exponent for a group $\G$ acting on metric space $(X,d)$ is defined by the exponential growth rate of 
$\{ \g\in \G \, |\, d(\g o , o )\leq R\} $  where  $o$ is  any point in $X$.  As Lorentzian manifolds are not metric spaces, we have to set a definition for the "distance" between two points.\\  
Hausdorff dimension is a numerical invariant showing how wild is the limit set and is defined through a metric on the boundary of the hyperbolic space. Here again for Lorentzian manifolds  we need to find a good class of metrics on the boundary. \\
Our work with D. Monclair \cite{glorieux2016} explains how these two invariants  can be generalised to GHMC manifolds (called AdS quasi-Fuchsian manifolds in the latter) and shows that we have equality between them. \\
Finally, let us mention that the work of D. Sullivan for hyperbolic manifolds shows  the critical exponent is related to the spectrum of the Laplacian. It has been studied for compact  $\AdS$ manifolds by F. Kassel and T. Kobayashi \cite{kassel2012discrete}, however nothing is known for GHMC manifolds. 

\bigskip 

In this article we will consider the exponential growth rate of the number of closed geodesics, which is the definition that can be translate in the Lorentzian setting in the most straightforward way.  

\subsection{Statement of results}
Let $M$ be  a GHMC $\AdS$ manifold and $S$ a Cauchy surface in $M$. 
Let $\C$ be the set of free homotopy classes of closed curves on $S$, this corresponds to the conjugacy classes of $\pi_1(S)$.  It is well known that if $S$ is endowed with a negatively curved metric there is an unique geodesic representative for any $c\in \C$. We will show in the next Section this is also true on $M$ endowed with its $\AdS$ metric. Hence, it gives for any $c\in\C$ a number $\ell_{Lor} (c)$ which is the length of the geodesic representative of $c$ in $M$. In a GH manifold there is neither  time nor lightlike closed geodesic, hence this length is positive. This is also the translation length of the conjugacy class of an element $\g\in \G$ associated to $c$. Let $\rho_L$ and $\rho_R$ be the two representations defining $M$. Recall that these representations are faithful and discrete in $\PSL_2(\R)$ henceforth they define two hyperbolic marked surfaces $S_L :=\Hyp^2/\rho_L(\G)$ and $S_R:=\Hyp^2/\rho_R(\G)$. The following  gives the relation between the length of a closed geodesic in $M$ and on the pair $(S_L,S_R)$. 
\begin{PropNoCount}[Proposition \ref{pr - lien longueur lorentzienne et hyp}, and \ref{pr - unique geodesic in free homotopy class}]
For every $c\in \C$, there is a unique geodesic representative of $c$ in $M$. Denoting by $\ell_{Lor} (c) $ its lorentzian length we have moreover. 
$$  \ell_{Lor}(c) =\frac{\ell_{S_L}(c) +\ell_{S_R}(c)}{2}.$$
\end{PropNoCount}
The proof  is two folded. First we need to show the uniqueness result which follows from the knowledge of the geodesics in $\AdS$. Then we need to compute its length, which results from an easy algebraic computation. It follows from this proposition that the counting problem in GHMC $\AdS$ manifolds is well defined. We define critical exponent by
$$\delta_{Lor}(M) :=\limsup_{T\tv \infty} \frac{1}{T}\log \Card \{c\in \C \, | \, \ell_{Lor} (c) \leq T\}.$$
The aim of this paper is the study of $\delta_{Lor}(M)$, when $(S_L,S_R)$ ranges over $\Teich(S)\times \Teich(S)$. 

Proposition \ref{pr - lien longueur lorentzienne et hyp} allows us to translate the purely Lorentzian problem into the counting problem on product of hyperbolic planes. Let us define for a pair of diffeomorphic marked hyperbolic surfaces $(S_L,S_R)$ the critical exponent by  
$$\delta(S_L,S_R):=\limsup_{T\tv \infty}  \frac{1}{T}\log \Card \{c\in \C \, | \, \ell_{S_L} (c)  +\ell_{S_R} (c)   \leq  T\},$$
where $\ell_* (c)$ designed the length of the unique closed geodesic in the free homotopy class of $c$ on $S_L$ or $S_R$. From Proposition \ref{pr - lien longueur lorentzienne et hyp}, it it clear that 
$$\delta_{Lor}(M) =2\delta(S_L,S_R).$$
We will then make the study of $\delta(S_L,S_R)$.

The invariant $\delta(S_L,S_R) $ has already been studied by C. Bishop and T. Steger in \cite{bishop1991three} where they showed the following rigidity result 
\cite{bishop1991three}
\begin{eqnarray}\label{th bishop steger}
\delta(S_0,S_1)\leq 1/2.
\end{eqnarray}
Moreover the equality occurs if and only if $S_0=S_1$. 

We can rephrase it in Lorentzian setting by, 
\begin{TheoremNoCount}[Theorem \ref{th - rigidité des GHMC}]
Let $M$ be a GHMC manifold parametrized by $(S_L,S_R)$, then 
$$\delta_{Lor}(M) \leq 1.$$ 
Moreover the equality occurs if and only if $S_L=S_R$.  
\end{TheoremNoCount}
This is the counterpart of Bowen's rigidity theorem for GHMC. We give a completly Lorentzian proof of this result in \cite{glorieux2016}. 

The counterpart for C. McMullen's Theorem is given by the following behaviour for $\delta_{Lor}$, they are explained with details in Section \ref{subsec - Examples}.
\begin{Example}[Section \ref{subsec - Examples}]
Let $S_0$ be a fixed point in $\Teich(S)$. 
\begin{enumerate}
\item Let $A$ be a pseudo-Anosov diffeomorphism on $S$ and call $S_n := A^{2n} S_0$. Then 
$$\lim_{n\tv \infty}\delta(S_0,S_n) = 0.$$
\item Let $S_n$ be the surface obtained after shrinking one simple closed geodesic on $S_0$. Then 
$$\liminf_{n\tv \infty} \delta (S_0,S_n) = \alpha >0 $$
\item Let $\tau $ be a Dehn twist around a simple closed curve and call $S_n := \tau^{2n} S_0$. Then 
$$\lim_{n\tv \infty}\delta(S_0,S_n) = \alpha<1/2 .$$
\item Let $S_t$ be a surfaces path obtained by Fenchel twist along a simple closed geodesic $c$. Then there exists a $2\ell_{S_0}(c)$-periodic  function $\delta $ such that: 
$$\lim_{t\tv \infty} | \delta(S_0,S_t) -\delta(t) | =0.$$
\end{enumerate}
\end{Example}

Finally, the counterpart Sanders' Theorem is our main theorem. It gives the isolation of the critical exponent $\delta_{Lor}$,
\begin{TheoremNoCount}[Theorem \ref{th - isolation 1 surface}]
Let $(M_n)$  be a sequence of GHMC manifolds, parametrized by $(S_0,S_n)$, where $S_0$ is a fixed hyperbolic surface, then 
$$\lim_{n\tv \infty}\delta_{Lor} (M_n) =1 \quad \text{iff} \quad \lim_{n\tv \infty} d_{Th}(S_0,S_n) =0,$$
where $d_{Th}$ is the (symmetrized) Thurston distance on the Teichmüller space. 
\end{TheoremNoCount}
Let $\Teich_\epsilon (S)$ be the thick part of Teichmüller space, that is the surfaces for which no closed geodesic has length less than $\epsilon>0$. A sequence $X_n\in \Teich(S)$ is said to stay in the thick part if there exists a $\epsilon>0$ such that $X_n\in \Teich_\epsilon(S)$ for all $n\in \N$.\\
We copy this definition and call the thick part of the set of GHMC $\AdS $ manifolds, a set for which the length of the smallest geodesic is bounded below by a fixed constant. From Proposition \ref{pr - lien longueur lorentzienne et hyp}, the sequence parametrized by $(S_0,S_n)$ in the previous Theorem stays in the thick part, since the length of the systole of the $\AdS$ manifold is bounded below by the systole of $S_0$ divided by 2. 

We deduce the following corollary thanks to Mumford's compactness Theorem and the invariance of critical exponent by diagonal action of the mapping class group 
\begin{TheoremNoCount}[Corollary \ref{cor - isolation 2 surfaces}]
Let $(M_n)$  be a sequence of GHMC manifolds, parametrized by $(S_n,S_n')$. If one of the sequences stays in the thick part of $\Teich(S)$ then
$$\lim_{n\tv \infty}\delta_{Lor} (M_n) =1 \quad \text{iff} \quad \lim_{n\tv \infty} d_{Th}(S_n,S_n') =0.$$
\end{TheoremNoCount}
Our hypothesis on the thick part is a bit more restrictive than just saying $(M_n)$ stays in the thick part of GHMC $\AdS$ manifolds ie. there is no closed Lorentzian geodesic of arbitrary small length. Indeed from \ref{pr - lien longueur lorentzienne et hyp} the length of a geodesic on $M_n$ goes to zero if and only if it goes to zero on both hyperbolic surfaces. 
Actually we found a counter example if $(M_n)$ does not stay in any thick part, cf Section \ref{sssection  - pinching at different speed}:
\begin{TheoremNoCount}[Theorem \ref{L'exposant critique tends vers 1/2 quand on pince une geodesique}]
There exists a sequence $(M_n)$ of GHMC manifolds, (not staying in the thick part as $\AdS$ manifolds) parametrized by $(S_n,S_n')$ such that
$$\lim_{n\tv \infty}\delta_{Lor} (M_n) =1 \quad \text{and} \quad \lim_{n\tv \infty} d_{Th}(S_n,S_n') =+\infty,$$
\end{TheoremNoCount}

We summarize the results previously cited in the following table : 
\begin{center}
\begin{tabular}{|c|c|}
  \hline
  \text{Quasi-Fuchsian setting } & GHMC AdS  setting.\\
  \hline
  $\delta_{QF} \in [1,2)$  & $\delta_{Lor} \in (0,1]$ \\
  $\delta_{QF}=1$ iff $S_0=S_1$ &   $\delta_{Lor}=1$ iff $S_0=S_1$.\\
  \cite{bowen1979hausdorff} & \textit{Theorem \ref{th - rigidité des GHMC}.}\\
  \hline 
  $\lim_{n\tv \infty}\delta_{QF} (S_0,A^n S_0)  = 2$ & $\lim_{n\tv \infty}\delta_{Lor}(S_0,A^{2n}S_0) = 0.$\\
  where $A$ is pseudo Anosov. &  where $A$ is pseudo Anosov.\\
  \cite{mcmullen1999hausdorff} & \textit{Proposition \ref{pr - example pseudo diff delta tends vers 0}.}\\
  \hline 
    $\lim_{n\tv \infty}\delta_{QF} (S_0,\tau^nS_0) =\alpha$ & $\lim_{n\tv \infty}\delta_{Lor}(S_0,\tau^{2n}S_0) = \alpha>0.$\\
 where $\tau $ is a Dehn twist along & where $\tau $ is a Dehn twist along  \\
 a simple  closed curve. &a simple  closed curve.  \\
 \cite{mcmullen1999hausdorff} & \textit{Section \ref{sssection Dehn twists}}\\  
    \hline
    $\lim_{t\tv \infty}\delta_{QF} (S_0,S_t) =\alpha<2$ & $\liminf_{n\tv \infty}\delta_{Lor}(S_0,S_t) = \alpha>0.$\\
  where $S_t$ is obtained by pinching &    where $S_t$ is obtained by pinching\\
  disjoint simple closed curves.  &  disjoint simple closed curves.\\
  \cite{mcmullen1999hausdorff} & \textit{Section \ref{sssection pinching one geodesic}}\\
  
 \hline
 $\lim_{t\tv \infty} | \delta_{QF}(S_0,S_t) - \delta(t) | =0$ &  $\lim_{t\tv \infty} | \delta_{Lor}(S_0,S_t) - \delta(t) | =0$ \\
 where $S_t$ is obtained by Fenchel twist &  where $S_t$ is obtained by Fenchel twist  \\
  along a simple closed curve $c$  &  along a simple closed curve $c$  \\
 and $\delta(t)$ is $\ell_0(c)$-periodic. &  and $\delta(t)$ is  $2\ell_0(c)$-periodic.\\
 \cite{mcmullen1999hausdorff} & \textit{Section \ref{sssection - fenchel nielsen twist}}\\
 \hline
 Let $M_n$ be a sequence of quasi-Fuchsian &   Let $M_n$ be a sequence of  GHMC AdS\\
 manifolds  parametrized by $(S_n,S_n')$  &    manifolds parametrized by $(S_n,S_n')$\\
staying in the thick part.   &  with one of the surface in thick part of $\Teich(S)$\\
Then : $\lim_{n\tv \infty}\delta_{QF} (M_n) =1 $  &  Then : $\lim_{n\tv \infty}\delta_{Lor} (M_n) =1 $  \\
 iff  $\lim_{n\tv \infty} d_{Teich}(S_n,S_n')=0 $ &   iff  $\lim_{n\tv \infty} d_{Th}(S_n,S_n')=0 $\\
 where $d_{Teich}$ is the Teichmüller distance. &  where $d_{Th}$ is the Thurston distance.\\
  \cite{sanders2014entropy} & \textit{Corollary \ref{cor - isolation 2 surfaces}}\\
\hline
\end{tabular}
\end{center}

Let us say a few words about the proof of Theorem \ref{th - isolation 1 surface}. There are mainly three ingredients in the proof. 
\begin{itemize}
\item A Anti-de Sitter geometry problem that we reduced to a problem on a product of hyperbolic planes. 
\item Understand critical exponent of pairs of Teichmüller representations acting on $\Hyp^2\times \Hyp^2$ through the Manhattan curve. 
\item See how Teichmüller representations degenerate thanks to earthquakes. 
\end{itemize}

\paragraph{Sketch of proof of Theorem \ref{th - isolation 1 surface}}
In Section 2, we see that the action on $\AdS$ can be reinterpreted in terms of the diagonal action on the product of two hyperbolic planes and we will then study the critical exponent $\delta(S_L,S_R)$ instead of $\delta_{Lor} (M)$. 

Then the proof of Theorem  \ref{th - isolation 1 surface} goes as follow. First, remark that one way is trivial and follows from the continuity of critical exponent. For the converse, let $(S_n)$ such that $\delta(S_0,S_n)\tv 1/2$. Two cases may happen: $(S_n)$ stays in a compact set or  $(S_n)$ goes out of every compact set of $\Teich(S)$.  For the first case, using rigidity theorem and continuity we easily show that $(S_n)$ must converge to $S_0$. We will show that the second case cannot happen. By Thurston's geology Theorem, we can connect $S_0$ to any point via an earthquake and we will show that along earthquake paths the critical exponent cannot tend to $1/2$. This will conclude the proof.

The idea to show that critical exponent cannot tend to $1/2$ along earthquake is to adapt the Dehn twist example. A Dehn twist $\tau$ is a particular case of earthquake and we can show using convexity of the geodesic length along earthquake paths that the critical exponent $\delta(S_0,\tau^{2n} S_0)$ is  decreasing, hence admits a limit which must be strictly less than $1/2$ by rigidity. The trick in this example is to symmetrize the length function : we show in fact
that $\ell(\tau^n c) +\ell(\tau^{-n} c)$ is increasing \emph{for all $c\in \C$}. From the invariance of $\delta (\cdot, \cdot) $ by the diagonal action of the mapping class group, it implies that $\delta(S_0,\tau^{2n} S_0)=\delta(\tau^{-n} S_0, \tau^n S_0)$ is decreasing. 

 We cannot use the same argument for general earthquakes. Indeed, there is no reason to have equality between $\delta(S_0,\EQ{\Lam}{2t}{S_0})$ and $\delta(\EQ{\Lam}{-t}{S_0},\EQ{\Lam}{t}{S_0})$ because an earthquake changes the length spectrum.  To take care of this difficulty we use a large deviation Theorem of geodesic flow, showing that "most"  curves on $S_0$
\footnote{A property $P$ is satisfied by most curves on $S_0$ if there is $\eta>0$ such that $\frac{\Card \{ c\in \C \, |\, \ell_0(c)\leq R \} \cap P }{\Card \{ c\in \C \, |\, \ell_0(c)\leq R \}} =o(e^{-\eta R } )$}
grows along earthquake. The next difficulty is to show a correlation between the behaviour of "most" curves on $S_0$ and "most" curves on the pair $(S_0,S_n)$\footnote{ We replace $ \{ c\in \C \, |\, \ell_0(c)\leq R \}$ by   $\{ c\in \C \, |\, \ell_0(c)+\ell_n(c)\leq R \}$ in the previous definition}. This will be done using what we call \emph{critical exponent with slope}, which is defined in a similar fashion as the critical exponent but where we impose the ratio  $\frac{\ell_n(c)}{\ell_0(c)}$ to be almost constant. The proof ends using estimates we found in Section 3 on critical exponents with slope and maximal slope.
\bigskip

The plan of the paper is the following. The next section is devoted to the Lorentzian geometry of GHMC AdS manifolds. The aim is to show how we can pass from a Lorentzian problem to a question on product of hyperbolic planes. We study the critical exponent for diagonal action on product of hyperbolic planes in  Section 3. We introduce the Manhattan curve as defined by M. Burger \cite{Burger} and make a careful study of the different invariants that can be  read out of  this curve: in particular, many inequalities between these invariants are deduced from the convexity of the Manhattan curve.  At the end of Section 3, we will prove as a by pass a generalisation of a Theorem of R. Sharp and R. Schwarz about pair of geodesics on two surfaces. The Section 4 is a description of geometric results about currents, laminations and earthquakes.  We will follow the presentation  of F. Bonahon \cite{BonahonGeometryofteichmuller} in order to get a clear picture of the results we need, namely, Corollary \ref{Geology}. This section contains also examples where we can compute either the limit of the critical exponent of some sequences, or at least get some bounds. The last section is devoted to the proof of the isolation Theorem  \ref{th - isolation 1 surface} and Corollary \ref{cor - isolation 2 surfaces}. 

\paragraph{Acknowledgements} This work is a part of my Ph. D. thesis and I am very grateful to Gilles Courtois for his support and advices. I would like to thank Maxime Wolff for his help about Teichmüller space, laminations, and Thurston compactification. I am also grateful to Jean-Marc Schlenker who suggested me the Anti-de Sitter interpretation of my work. 
Finally I would like to thank the referees for many useful comments. 

\section{From Lorentzian to hyperbolic geometry}

As we said in the introduction, the geometry of GHMC, $\AdS$ manifolds is encoded through Mess parametrization by two Fuchsian representations,  ie. faithful and discrete representations in $\PSL_2(\R)$. The aim of the section is to recall some basic facts about $\AdS$  geometry and to explain how closed geodesics on Lorentzian manifolds are related to the closed geodesics on the hyperbolic surfaces used in Mess parametrization. 

Our main result in this section is the Theorem \ref{th - rigidité des GHMC}, which is the counterpart of Bowen's Theorem for globally hyperbolic manifolds. 
\paragraph{Proof of Theorem \ref{th - rigidité des GHMC}.} First we compute the length of a closed geodesic in the $\AdS$ manifold parametrized by two Fuchsian representations, in terms of the lengths of closed geodesics in the corresponding hyperbolic surfaces: Proposition \ref{pr - lien longueur lorentzienne et hyp}. Then we prove that there is a unique closed geodesic in every non trivial isotopy class of closed curve: Proposition \ref{pr - unique geodesic in free homotopy class}. Applying the Theorem of Bishop-Steger \cite{bishop1991three} we get Theorem \ref{th - rigidité des GHMC}.

\subsection{Model for AdS space}
The model for $\AdS$ we will use is the projective model \cite{barbot2007constant}, namely $\AdS\simeq \PSL_2(\R)$ endowed with the restriction of the quadratic form $-\det \left(\begin{array}{cc}
a & b\\
c & d \\
\end{array} \right)= -ad+bc$. The isometry group of $\AdS$ is $\PSL_2(\R)\times \PSL_2(\R)$ ; an element $\g=(A,B) \in \PSL_2(\R)\times \PSL_2(\R)$ acts on $X\in \AdS$ by, $\g \cdot X = A X B^{-1}$.
The topological boundary of $\AdS$ is  the set of projective non-invertible matrices.  It is parametrized by $\RP^1_L \times \RP^1_R$, through the Segre embedding :  to $u_L= \left[ \begin{array}{c}
x_L \\ 
y_L
\end{array}  \right]$ 
and 
$u_R= \left[ \begin{array}{c}
 x_R\\ 
y_R
\end{array}  \right]$ 
we associated the projective non-invertible matrix $u_Lu_R^t =\left[ \begin{array}{cc}
x_Lx_R & x_L y_R \\ 
y_L x_R & y_L y_R
\end{array} \right]$
The action on $\AdS$ extends on a continuous action on the boundary given for   $\g =(A,B)\in \PSL_2(\R)\times \PSL_2(\R)$ by $\g \cdot u_Lu_R^t = Au_L u_R^t B^{-1} = Au_L (B^*u_R)^t$, where $B^* =  (B^{-1})^t$.

\subsection{Geodesics of $\AdS$}
In a Lorentzian manifolds we can defined geodesics in the same manner as in a Riemannian manifold by parallel transport with respect to the Levi-Civita connection. In our model of $\AdS$ space there is a simple description given in \cite{barbot2007constant}. Namely, geodesics in $\AdS$ are the intersection of projective lines and $\AdS$. The type (timelike, null or spacelike) is invariant by left and right multiplication by elements of $\PSL_2(\R)$. It is determined by the signature of $-\det$ restricted to the plane defining the projective line.   It is then easy to see that the one passing through identity  are given by  1-parameter subgroups of $\PSL_2(\R)$. 
\begin{itemize}
\item They are time-like if the corresponding 1-parameter subgroup is elliptic, ie conjugated to 
$\left[  \begin{array}{cc}
\cos(\theta) & \sin(\theta) \\ 
-\sin(\theta)  & \cos(\theta)
\end{array} \right].$ These geodesics are entirely contained in $\AdS$.  

\item They are light-like if the corresponding 1-parameter subgroup is parabolic, ie conjugated to $\left[  \begin{array}{cc}
1 & x \\ 
0  & 1
\end{array} \right].$ These geodesics are tangent to $\partial AdS$.  

\item They are space-like if the corresponding 1-parameter  subgroup is hyperbolic, ie conjugated to 
$\left[  \begin{array}{cc}
 e^{t} & 0 \\ 
0 & e^{-t} 
\end{array} \right].$ 
These geodesics have two endpoints in $\partial AdS$ 
$x(+\infty)=\left[  \begin{array}{cc}
1 & 0 \\ 
0  & 0
\end{array} \right] =\left[ \begin{array}{c}
 1\\ 
0
\end{array}  \right] \times  \left[ \begin{array}{c}
 1\\ 
0
\end{array}  \right]^t $  and  $x(-\infty)=\left[  \begin{array}{cc}
0 & 0 \\ 
0 & 1
\end{array} \right] =\left[ \begin{array}{c}
 0\\ 
1
\end{array}  \right] \times  \left[ \begin{array}{c}
 0\\ 
1
\end{array}  \right]^t  .$
\end{itemize}
Let $\g\in \PSL_2(\R)\times \PSL_2(\R)$  be given by
 $\left( \left[  \begin{array}{cc}
e^\lambda & 0\\ 
0  & e^{-\lambda} 
\end{array} \right],\left[  \begin{array}{cc}
e^{-\mu} & 0\\ 
0  & e^{\mu} 
\end{array} \right] \right), \lambda,\mu>0,$ 

A computation shows that the element $g$ fixes the geodesic  $\left[  \begin{array}{cc}
 e^{t} & 0 \\ 
0 & e^{-t} 
\end{array} \right]_{t\in \R},$ and acts by translation on this geodesic. 
 
\begin{eqnarray*}
g\cdot \left[  \begin{array}{cc}
1 & 0\\ 
0  & 1 
\end{array} \right]  &=&  \left[  \begin{array}{cc}
e^\lambda & 0\\ 
0  & e^{-\lambda} 
\end{array} \right] \left[  \begin{array}{cc}
1 & 0\\ 
0  & 1 
\end{array} \right] \left[  \begin{array}{cc}
e^{-\mu} & 0\\ 
0  & e^{\mu} 
\end{array} \right] ^{-1}\\
&=& \left[  \begin{array}{cc}
e^{\lambda+\mu} & 0\\ 
0  & e^{-(\lambda+\mu)}
\end{array} \right]
\end{eqnarray*} 

We can then define the attracting fixed point to be  $$g^{+} = x(+\infty)=\left[  \begin{array}{cc}
1 & 0 \\ 
0  & 0
\end{array} \right] $$
and the repelling fixed point to be 
$$g^{-} = x(-\infty)=\left[  \begin{array}{cc}
0& 0 \\ 
0  & 1
\end{array} \right].$$

Remark that the geodesic $t\tv  \left[  \begin{array}{cc}
 e^{t} & 0 \\ 
0 & e^{-t} 
\end{array} \right]$ is parametrized by unit speed. Hence the translation length for the induced metric on this geodesic is equal to $\lambda+\mu$. 

Recall that any hyperbolic element of $\PSL_2(\R)$ is conjugated to $  \left[  \begin{array}{cc}
e^\lambda & 0\\ 
0  & e^{-\lambda} 
\end{array} \right] $,  hence a pair of hyperbolic elements $(g_L,g_R)$ acting on $\AdS$ has a well defined space like axis.

\begin{definition}
We call axis of a pair of hyperbolic elements $g=(g_L,g_R)$ acting on $\AdS$ the unique spacelike geodesic joining $g^-$ to $g^+$
\end{definition}
Let us remark that this is not the unique geodesic fixed by $g$. Indeed for the element $\left( \left[  \begin{array}{cc}
e^\lambda & 0\\ 
0  & e^{-\lambda} 
\end{array} \right],\left[  \begin{array}{cc}
e^{-\mu} & 0\\ 
0  & e^{\mu} 
\end{array} \right] \right),$
the geodesic 
$\left[  \begin{array}{cc}	
0 &e^t\\ 
-e^{-t} & 0
\end{array} \right]_{t\in \R},$ is also preserved. This geodesic joins the attracting fixed point of $ \left[  \begin{array}{cc}
e^\lambda & 0\\ 
0  & e^{-\lambda} 
\end{array} \right],$ to the repelling fixed point of $ \left[  \begin{array}{cc}
e^{-\mu} & 0\\ 
0  & e^{\mu} 
\end{array} \right]$ and therefore does not have a satisfying geometric meaning. 
We will indeed show in the next section that for a GHMC manifolds this geodesic is not in the range of the developing map.\\

As in the hyperbolic setting, we define the translation length of $g$ :
\begin{definition}
We call translation length of an isometry $g$ consisting of a pair of hyperbolic elements $g=(g_L,g_R)$, and we denote $\ell_{Lor}(g)$ the distance (for the induced metric) between $g$ and $g\cdot o$ for any point $o$ on the axis of $g$. 
\end{definition}

The Lorentzian translation length is invariant by conjugation  as well as the hyperbolic translation. Moreover this last one for 
$\left[  \begin{array}{cc}
e^\lambda & 0\\ 
0  & e^{-\lambda} 
\end{array} \right] $ is equal to  $2\lambda$.  We have proved the following 
\begin{proposition}\label{pr - lien longueur lorentzienne et hyp}
Let $\g=(\g_1,\g_2)$ be an isometry of $\AdS$ defined by a pair of hyperbolic elements  then $\ell_{Lor}(\g) =\frac{\ell_{\Hyp}(\g_1)+\ell_{\Hyp}(\g_2)}{2}.$
\end{proposition}

In Subsection \ref{subsec - closed geodesics on GHMC mnaifolds}, we are going to prove the 1-to-1 correspondence between the set of free isotopy classes of closed curves on Cauchy surfaces and closed geodesics in globally hyperbolic manifolds.

\subsection{Closed geodesics on GHMC manifolds}\label{subsec - closed geodesics on GHMC mnaifolds}

Our aim is to show that in any free homotopy class of closed curve, there is exactly one geodesic. 
A GHMC, $M$, is topogically the product $\R\times S$, hence the set $\C$ of free homotopy classes of closed curves on $M$ is the same as the set of free homotopy classes of closed curves on $S$.
The proof relies on the knowledge of geodesics in the range of the developing map,  $D : \tilde{M}\tv \AdS$. 

The only closed geodesics of $\AdS$ are timelike, and it is proven in \cite[Corollary 5.20 and Remark 5.16 ]{barbot2007constant} that $D (\tilde{M})$ does not contain any timelike geodesic. Hence it does not contain any closed geodesic. 

From the work of  G. Mess in \cite{mess2007lorentz}, we know that the holonomy $\rho$ is given by the product of two Fuchsian representations. Since $S$ is supposed to be compact, it follows that for every $\g\in \G$, $\rho(\g) =(\g_1,\g_2) \in \PSL_2(\R)\times \PSL_2(\R)$  is a pair of hyperbolic isometries. 
We have seen that such an isometry fixes two geodesics in $\AdS$, we have to show  that only one of them is in $D(\tilde{M})$, which is the axis of $\rho(\g)$. 

For this, recall that we defined the axis as the geodesic joining the attracting and the repelling point of $\rho(\g)$, those points are in $\Lambda$ the limit set of $\G$ on $\AdS$. As before, without loss of generality, we can conjugate $\rho$ such that $\rho(\g) = \left(\left[\begin{array}{cc}
e^\lambda & 0\\ 
0  & e^{-\lambda} 
\end{array} \right],\left[  \begin{array}{cc}
e^{-\mu} & 0\\ 
0  & e^{\mu} 
\end{array} \right] \right).$
The axis $\left[\begin{array}{cc}
e^t & 0\\ 
0  & e^{-t} 
\end{array} \right]_{(t\in \R)}$ is in the convex core of the limit set, since once again $\left[\begin{array}{cc}
1 & 0\\ 
0  & 0
\end{array} \right] $ and $\left[\begin{array}{cc}
0& 0\\ 
0  & 1
\end{array} \right]$ are points of the limit sets.
In particular, this implies that $Id =\left[\begin{array}{cc}
1 & 0\\ 
0  & 1
\end{array} \right]$ is in the convex core. From Lemma 6.17 of \cite{barbot2007constant}, for every point in the convex core, its dual plane is disjoint from the $D (\tilde{M})$. The dual plane of a point $p$ is the intersection of the orthogonal for the quadratic form with $\SL_2(\R)$, which is for $Id$ the plane $Tr(X)=0$. I follows that the other geodesic fixed by $\rho(\g)$ : $\left[\begin{array}{cc}
0 & e^t\\ 
-e^{-t}   & 0
\end{array} \right]_{(t\in \R)}$, is not contained in $D(\tilde{M})$. 

Finally, we conclude the proof as  Lemma B.4.5 of \cite{benedetti2012lectures}, that we reproduce for the convenience of the reader. 
\begin{proposition}\label{pr - unique geodesic in free homotopy class}
In every non trivial free homotopy class of closed curves in $M$ there is a unique geodesic. 
\end{proposition}
 
\begin{proof}
Let $\alpha$ be a non trivial closed loop. Let $\g$ be a representative in $\pi_1(M)$ of the class of $\alpha$. We have seen that $\rho(\g)$ fixes exactly one geodesic $\tilde{g}$ in $D(\tilde{M})$. The projection of $\tilde{g}$ is a geodesic loop which is in the same class as $\alpha$. 

Conversely, let $g_1$ be a geodesic loop in $M$ representing the same class as $\alpha$.  Let $\tilde{g_1}$ denote a lift, and $\g_1$ the element in $\G$ such that $\rho(\g_1)(\tilde{g_1}(0))=\tilde{g_1}(1)$. Since $\tilde{g_1}$ is a geodesic it follows that it is invariant by $\rho(\g_1)$. Moreover, there exists $h$ such that $\rho(\g_1)=\rho(h^{-1} \g h)$, whence $\tilde{g_1}= \rho(h^{-1} ) (\tilde{g})$ and from uniqueness of the geodesic fixed by $\rho(\g)$, we have $\tilde{g_1}= \tilde{g}.$
\end{proof}

The proof and Proposition \ref{pr - lien longueur lorentzienne et hyp}  shows that for any $c\in \C$ there is a unique geodesic whose length is equal to $\ell_{Lor} = \frac{\ell_L(c)+\ell_R(c)}{2}$.

We then defined critical exponent for GHMC $\AdS$ manifolds :
$$\delta(M) :=\limsup_{R\tv \infty} \frac{1}{R} \log\Card \{ c\in \C \, | \, \ell_{Lor}(c) \leq R\}.$$
Moreover, if $M$ is parametrized by $(\rho_L,\rho_R)$ we have 
\begin{eqnarray}
\delta(M) &=&\limsup_{R\tv \infty} \frac{1}{R} \log \Card \{ c\in \C \, | \, \frac{\ell_L(c)+\ell_R(c)}{2} \leq R\}\\
				&=&2 \limsup_{R\tv \infty} \frac{1}{R} \log \Card \{ c\in \C \, | \, \ell_L(c)+\ell_R(c) \leq R\}
\end{eqnarray}

Now using Bishop-Steger Theorem  \cite{bishop1991three} we prove the result equivalent to Bowen's result for quasi-Fuchsian manifold. 

\begin{theorem}\label{th - rigidité des GHMC}
Let $M$ be a GHMC manifold parametrized by $(S_L,S_R)$. We have
$$\delta_{Lor}(M) \leq 1$$ 
Moreover the equality occurs if and only if $S_L=S_R$. 
\end{theorem}

The aim of the rest of this paper is to study $\delta$ when $\rho_L$ and $\rho_R$ move in $\Teich(S)$.  We are not going to use the Lorentzian interpretation anymore.\footnote{ With D. Monclair, we propose a  Lorentzian proof of the previous result independent of the one of Bishop Steger in \cite{glorieux2016}. } From now on, we will study the distribution of closed geodesics on the two hyperbolic surfaces associated to $\rho_L$ and $\rho_R$. This corresponds geometrically to the diagonal action on $\Hyp^2\times \Hyp^2 $
endowed with the Manhattan metric, $d_M $, which is the sum of the two hyperbolic metrics. The first tool we are going to introduce is the Manhattan curve defined by M. Burger in \cite{Burger} for which we will replace, for the sake of coherence,  the "left" and "right" notations by "1" and "2", since it is a curve in $\R^2$, for example Manhattan metric will be denoted by $d_M=d_1+d_2$.

Manhattan curve is a  nice tool which allow to visualize different invariants associated to two Fuchsian representations. We can easily read on this curve the critical exponent for the diagonal action,  
$$\delta(S_1,S_2) =  \limsup_{R\tv \infty} \frac{1}{R} \log \Card \{ c\in \C \, | \, \ell_1(c)+\ell_2(c) \leq R\},$$ which is also equal to the exponential growth of the cardinal of an orbit in $\Produit$ \footnote{this  follows from a small modifications of the arguments in \cite{Knieper}}  :
$$\delta(S_1,S_2) =\delta(\rho_1,\rho_2) =  \limsup_{R\tv \infty} \frac{1}{R} \log \Card \{ \g\in \G \, | \, d(\rho_1(\g) o, o) + d(\rho_2( \g ) o ,o ) \leq R\}.$$ This last invariant is actually the one which is called \emph{critical exponent} usually. Moreover on the Manhattan curve, we can read out what we will call \emph{critical exponent with slope}, the \emph{maximal slope}, and \emph{geodesic stretch}. These three  invariants are related by different means to the ratio of the geodesic lengths on the two surfaces $\frac{\ell_2(c)}{\ell_1(c)}$. From the convexity of Manhattan curve it is easy to deduce inequalities between all of them.

\section{Manhattan curve}\label{sec - Manhattan curve}

Let $S$ be a compact surface of genus $g\geq 2$, $\G:=\pi_1(S)$ its the fundamental group, $\rho_1$, $\rho_2$, two Fuchsian representations of $\G$ and $\C$ the set of free homotopy classes of closed curves. 
We will denote by $\G_i := \rho_i(\G)$ the corresponding subgroups of $\PSL(2,\R)$ and $S_i := \Hyp^2 / \G_i$ the hyperbolic surfaces homeomorphic to $S$. We will often forget $\rho_i$ and denote $\rho_i (\g)$ by $\g_i$. We look at the diagonal action on $\Hyp^2\times \Hyp^2$, ie. $\g.(x,y) = (\rho_1(\g) x, \rho_2(\g) y)$, that we endow with different Manhattan distances, defined by the weighted sum of the two hyperbolic distances on each factor, $d_M^{x,y} := xd_1 +yd_2$. Since $x,y$  can be negative it is not necessarily genuine distances, however we will often restrict ourselves to $x,y\in (0,1)$.  Let $o=(p,q) \in \Produit $ be a point fixed once for all.
\begin{definition}\label{def Poincaré series}
 The \emph{Poincaré series } is defined by 
$$P_M[\rho_1,\rho_2,x,y](s) := \sum_{\g\in \G} e^{-s d_M^{x,y}(\g o,o)   }.$$
\end{definition}

There is a convenient way to normalize the pair $(x,y)$ which is the definition of the Manhattan curve :
\begin{definition}\label{def Manhattan curve}
The Manhattan curve is defined by 
$$C_M := \{(x,y) \in \R^2 \, | \, \text{The abscissa of convergence of } P_M[\rho_1,\rho_2,x,y](s) \text{ is 1}\}.$$
\end{definition}

Recall the definition of the critical exponent associated to the pair $(\rho_1,\rho_2)$.
\begin{definition}\label{def critical exponent via poincaré serie}
We define the \emph{critical exponent of $\rho_1$ and $\rho_2$} by 
$$ \delta(\rho_1,\rho_2) := \inf \{ s >0 \, | \, P_M[\rho_1,\rho_2,1,1](s)<\infty \}.$$
\end{definition}

By triangular inequality $ \delta(\rho_1,\rho_2)$  depends neither on $o\in \Produit$, nor on the conjugacy class of $\rho_1$ and $\rho_2$, hence it defines a function on $\Teich(S) \times \Teich(S)$, still denoted by $\delta$. 

We could have defined it in the following equivalent way : 
\begin{proposition}
The following are equal :
\begin{enumerate}
\item $\delta(\rho_1,\rho_2)$ as defined in Definition \ref{def critical exponent via poincaré serie}.
\item $\limsup_{R\tv \infty} \frac{1}{R} \log \Card \{ \g\in \G \, | \, d(\rho_1(\g) o, o) + d(\rho_2( \g ) o ,o ) \leq R\}$.
\item The abscissa of convergence of $\sum_{c\in \C} e^{-s (\ell_1(c) +\ell_2(c))}.$
\item $\limsup_{R\tv \infty} \frac{1}{R} \log \Card \{ c\in \C \, | \, \ell_1(c)+\ell_2(c) \leq R\}$.
\end{enumerate}
\end{proposition}

From the geometric nature of points 3 and 4, we choose to endow the Teichmüller space of $S$ with the Thurston distance. 
\begin{definition}\cite{thurston1998minimal}
 Let $(S_1,S_2)$ be  two hyperbolic surfaces. The Thurston distance is defined by 
$$d(S_1,S_2) =\log \sup_{c\in \C}\max \left( \frac{\ell_2(c)}{\ell_1(c)} ; \frac{\ell_1(c)}{\ell_2(c)}\right)$$
\end{definition}
We define $\dil^+ := \sup_{c\in \C} \frac{\ell_2(c)}{\ell_1(c)}$ and $\dil^-:= \inf  \frac{\ell_2(c)}{\ell_1(c)}$, therefore $d(S_1,S_2)= \log \max \left(\dil^+, \frac{1}{\dil^-}\right).$

Two remarks have to be pointed out. First, we choose the symmetric distance instead of the usual asymmetric because our isolation result is stronger in this way. Second, it is usually defined taking the supremum over simple closed curves, but from Thurston's work \cite[Proposition 3.5]{thurston1998minimal}, the two definitions coincide.

We begin this section by some inequalities between critical exponent and intersection number then we study the correlation number with slope.
The intersection between two hyperbolic surfaces will  be defined in Definition \ref{def - intersection entre deux surfaces}. This is a generalisation of the classical geometric intersection between two closed curves. In this section, the precise definition does not matter, we will just prove some inequalities coming from the Theorem of M. Burger \cite{Burger} about the Manhattan curve. For all this section we fix $S_1$ and $S_2$ two hyperbolic surfaces and we  call $\gs$ the geodesic stretch function on $\Teich(S)\times \Teich(S)$ defined by, $\gs : (S_1,S_2) \tv \frac{i(S_1,S_2)}{i(S_1,S_1)}$, where $i(S_i, S_j)$ is the intersection number between  $S_i$ and $S_j$. The function $\gs$ has a natural geometric interpretation, it corresponds to the ratio $\frac{\ell_2(c)}{\ell_1(c)}$ for "typical curves" on $S_1$. More precisely, if $(c_n)$ is a sequence of closed curves on $S_1$ which becomes equidistributed with respect to the Liouville measure on $T^1S_1$, we have $\gs(S_1,S_2) = \lim\frac{\ell_2(c_n)}{\ell_1(c_n)}$. Finally let $\lambda(x)$ be the slope of a normal vector to $\C_M$ at the abscissa $x$. Recall the following, 
\begin{theorem}\cite[Theorem  1]{Burger}\label{Rigidite}
 The Manhattan curve is  the straight line containing $(1,0)$ and $(0,1)$ if and only if $S_1$ and $S_2$ are equal in $\Teich(S)$. Moreover 
\begin{itemize}
\item $\lambda(1)=\gs(S_1,S_2)$.
\item $\lim_{x\tv +\infty}  \lambda(x) = \dil^+$, $\lim_{x\tv -\infty}  \lambda(x) = \dil^-$.
\end{itemize}
\end{theorem} 

The other result we will need is the 
\begin{theorem}\cite{Sharp}
The Manhattan curve is real analytic. 
\end{theorem}

%
For concision we write $\delta$ for $\delta(S_1,S_2)$ and $\gs$ for   $\gs(S_1,S_2)$ .

We first begin by recalling some basic facts about $\delta$ and $\C_M$. 
\begin{proposition}\label{pr - basic properties}
\begin{enumerate}
\item  The points $(1,0)$ and $(0,1)$  are on $\C_M$. 
\item  The  intersection point between $\C_M$ and the line $y=x$ has coordinates $(\delta, \delta)$.
\item The Manhattan curve is convex and strictly convex if $S_1\neq S_2$.
\item If $S_1\neq S_2$ then $\lambda : \R \tv (\dil^-,\dil^+)$ is one-to-one. 
\end{enumerate}
\end{proposition}

\begin{proof}

\begin{enumerate}
\item Follows from the compactness of $S$. Indeed the critical exponent of $\sum_{c\in \C} e^{-sl(c)}$ is $1$ for a compact surface. 
\item The intersection point has coordinates of the form $(x,x) $ since it is on the line $y=x$. Since $(x,x)$ is on $\C_M$, $\sum_{c\in \C} e^{-s x(\ell_1(c)+\ell_2(c))}$ has critical exponent equal to $1$, this exactly means that $x=\delta$ 
\item Follows from the convexity of the exponential map. More precisely let $(x_1,y_1)$, $(x_2,y_2)$ be two points of $\C_M$ then by Hölder's inequality 
$$P[S_1,S_2,tx_1+(1-t)x_2, ty_1+(1-t)y_2](s) \leq (P[S_1,S_2,x_1, y_1](s))^t (P[S_1,S_2,x_2, y_2](s))^{1-t}.$$ By definition, both series of the right hand side have critical exponent equal to $1$, hence $P[S_1,S_2,tx_1+(1-t)x_2, ty_1+(1-t)y_2](s)$ has critical less than $1$, which exactly means that $(tx_1+(1-t)x_2, ty_1+(1-t)y_2)$ is above $\C_M$. The strict convexity follows from the analiticity of $\C_M$.
\item By strict convexity, $\lambda$ is strictly increasing and it is  continuous hence one-to-one. 
\end{enumerate} 
\end{proof}

The convexity  of $\C_M$ and Theorem  \ref{Rigidite} imply the rigidity Theorem of C. Bishop and S. Steger we mentioned in the introduction, $\delta \leq \frac{1}{2}$ with equality if and only if $S_1=S_2$.  

The last statement in Proposition \ref{pr - basic properties}, shows that for every $\lambda \in  (\dil^-,\dil^+)$ there is a point on the curve for which the slope of the normal is exactly $\lambda$. We set the following notation
\begin{definition}\label{def - x(lambda)}
The point on the Manhattan curve for which the slope of the normal is $\lambda\in (\dil^-,\dil^+)$, will be noted $(x(\lambda),y(\lambda))\in \C_M$.
\end{definition}

\begin{definition}\label{def maximal slope}
We call \emph{maximal slope} associated to $S_0$ and $S_1$ the slope of a normal vector at $(\delta,\delta)$. This is $\lambda(\delta)$. 
\end{definition}

 By convexity $\C_M$ is above the line $y = \frac{-1}{\lambda(\delta)} (x-\delta) +\delta$. Since  $(0,1)\in \C_M$ and $(1,0)\in \C_M$, we obtain the two following inequalities : 
$$ \frac{\delta}{1-\delta} \leq \lambda(\delta) \leq \frac{1-\delta}{\delta}.$$

We will use the following corollary of these inequalities.  
\begin{corollaire}\label{Si delta tends vers 1 lambda tends vers 1}
Let $S_n$ and $S_n'$ be two sequences of hyperbolic surfaces. If $\, \lim_{n\tv \infty} \delta(S_n,S_n') =\frac{1}{2}$ then $\lim_{n\tv \infty}\lambda(\delta(S_n,S_n')) = 1$.
\end{corollaire}

By convexity again, the line $y = \frac{-1}{\gs} (x-1)$ is below $\C_M$. Hence taking the intersection with $y=x$ we get $\delta \geq \frac{-1}{\gs} (\delta-1)$, which is equivalent to for every $S_0,S_1\in \Teich(S)$ : 
$$\delta(S_0,S_1)\geq \frac{1}{1+\gs(S_0,S_1)}.$$
This gives the following,
\begin{corollaire}\label{Si I tends vers 1 delta tends vers 1/2}
Let $S_n$ and $S_n'$ be two sequences of hyperbolic surfaces. If $ \, \lim_{n\tv \infty}  \gs(S_n,S_n')=1$, then $\lim_{n\tv \infty} \delta(S_n,S_n') =\frac{1}{2}$ .
\end{corollaire}
Moreover, $\delta \geq \frac{1}{1+\gs}$ gives, $\gs \geq \frac{1}{\de}-1$, hence we have the following corollary which is in the paper of M. Burger, 
\begin{corollaire}\cite[Corollary 1]{Burger}\label{l'intersction est plus grande que }
We have $\gs(S_1,S_2)\geq 1$, with equality if and only if $S_1=S_2$.
\end{corollaire}

\subsection{Critical exponent with slope}

A central idea in our work is to study the proportionality factors between the two lengths of the geodesic corresponding  to a closed curve on $S_1$ and $S_2$. 

The principal result in this Section is a formula for the critical exponent with slope $\lambda$ in terms of the Manhattan curve. 

For $\lambda\in (\dil^-,\dil^+)$ and $\epsilon>0$, let 
\begin{equation}\label{eq - def de CS(lambda epsilon)}
\C(\lambda,\epsilon) := \left\{ c \in \C  \, | \, \left| \frac{\ell_2(c) }{\ell_1 (c)}  - \lambda \right| \leq \epsilon \right\}.
\end{equation}

To this set is naturally associated a critical exponent, namely 
$$\delta(S_1,S_2,\lambda,\epsilon) := \inf\left\{s>0 \, | \, \sum_{c\in \C(\lambda,\epsilon)} e^{-s(\ell_1(c)+\ell_2(c)) } < +\infty\right\}.$$
\begin{definition}\label{def - critical exponent with slope}
 The \emph{critical exponent with slope $\lambda$}, is defined by 
 $$\delta(S_1,S_2,\lambda) := \lim_{\epsilon \tv 0} \delta(S_1,S_2,\lambda, \epsilon).$$
\end{definition}
As the critical exponent we could have defined it by 
$$\delta(S_1,S_2,\lambda) = \lim_{\epsilon\tv 0} \limsup_{T\tv +\infty}\frac{1}{T} \log \Card \{ c\in \C(\lambda,\epsilon) \, |\, \ell_1(c)+\ell_2(c) \leq T\}.$$

\begin{theorem}\label{rigidite de l exposant critique directionel}
Let $(x,y)\in \C_M$, and $\lambda \in (\dil^-,\dil^+)$. We have the following inequality $$\delta(\lambda) \leq \frac{x+\lambda y}{1+\lambda},$$
 equality occurs if and only if $(x,y)=(x(\lambda),y(\lambda))$, the point on the Manhattan curve for which the slope of the normal vector is equal to $\lambda$, see Definition \ref{def - x(lambda)}.
\end{theorem}

First we are going to prove the inequality, which is a simple algebraic manipulation. This will be done in the following Lemma \ref{inegalitépourdeltalambda}. The equality case is a bit harder and we delay its proof for the next subsection. The Theorem \ref{rigidite de l exposant critique directionel} will finally be proven after Corollary \ref{relatioC DELTA et courbe de manhattan}.

 This kind of results implicitly appears, but for Riemannian metric, in the work of G. Link \cite{link2004hausdorff}. We propose a proof since our context is a little bit different and since recognizing that \cite[Lemma 3.7]{link2004hausdorff} and  Lemma \ref{inegalitépourdeltalambda} from our work are similar is not  obvious. 
\begin{lemme}\label{inegalitépourdeltalambda} Compare to \cite[Lemma 3.7]{link2004hausdorff}.
For every $(x,y)\in \C_M$ and every $\lambda\in (\dil^-,\dil^+)$ we have the inequality $\delta(\lambda) \leq \frac{x+\lambda y}{1+\lambda}$.
\end{lemme}
\begin{proof}
 Let $(x,y)\in \C_M$, $\lambda \in  (\dil^-, \dil^+)$, and $ c\in \C(\lambda,\epsilon)$. First we suppose that $x>0$ and $y>0$. Then 
\begin{eqnarray}
 x\ell_1(c)+y\ell_2(c) &=& \left( x \frac{1}{1+\ell_2(c)/\ell_1(c)} + y \frac{\ell_2(c)/\ell_1(c)}{1+ \ell_2(c)/\ell_1(c)} \right) \left( \ell_1(c) +\ell_2(c)\right) \\\label{inegalité 2 interne à la preuve du lemme de Link}
&\leq & \frac{x + (\lambda +\epsilon) y}{1+ \lambda -\epsilon} \left( \ell_1(c) +\ell_2(c) \right) 
\end{eqnarray}
This implies for  the Poincaré series that  : 
\begin{eqnarray}
\sum_{c \in \C} e^{-s(x\ell_1(c) + y\ell_2(c)) } &\geq &   \sum_{c \in \C(\lambda, \epsilon)} e^{-s(x\ell_1(c) + y\ell_2(c)) } \\
 &\geq & \sum_{c \in \C(\lambda, \epsilon)} e^{-s\left(\frac{x + (\lambda +\epsilon) y}{1+ \lambda -\epsilon} \right)\left( \ell_1(c) +\ell_2(c) \right)  } 
\end{eqnarray}
Since $(x,y)\in \C_M$ the critical exponent of the left hand side is equal to 1, therefore
\begin{eqnarray}\label{inegalités 8 interne à la preuve de Link}
1 &\geq& \delta(\lambda,\epsilon)\left( \frac{1+ \lambda -\epsilon}{x + (\lambda +\epsilon) y}\right)\\ 
\delta(\lambda,\epsilon) &\leq & \frac{x+ (\lambda +\epsilon) y}{1+ \lambda -\epsilon}\\
\text{passing to the limit gives,} \quad 
 \label{limite interne Link}
 \delta(\lambda) &\leq & \frac{x + \lambda y}{1+ \lambda }.
\end{eqnarray}
This end the proof for $x>0$ and $y>0$. 

If $x<0$ and $y>0$, the inequality (\ref{inegalité 2 interne à la preuve du lemme de Link}) would become 
$$ x\ell_1(c)+y\ell_2(c)  \leq \left( \frac{x}{1+ \lambda +\epsilon}  +\frac{ (\lambda +\epsilon) y}{1+ \lambda -\epsilon} \right) \left( \ell_1(c) +\ell_2(c) \right) $$
and then the inequality (\ref{inegalités 8 interne à la preuve de Link}) would become 
$$1 \geq \delta(\lambda,\epsilon)\left( \frac{x}{1+ \lambda +\epsilon}  +\frac{ (\lambda +\epsilon) y}{1+ \lambda -\epsilon} \right).$$
Passing to the limit as in  \ref{limite interne Link} ends the proof in the case $x<0$ and $y>0$. 

The case  $x>0$ and $y<0$ can be treated similarly.
\end{proof}

%
%

The equality case in Theorem \ref{rigidite de l exposant critique directionel} will be our goal for the last part of this section and proven in Corollary \ref{relatioC DELTA et courbe de manhattan}. 

\subsection{Correlation number with slope}\label{sec:correlation number}
We are actually going to show the equality case using  
an extension of a formula due to R. Sharp, \cite{Sharp}  about the correlation number. This result of R. Sharp uses thermodynamic formalism for the geodesic flow. We will make a brief survey of results on thermodynamic formalism and geodesic flow which ends with Theorems \ref{Formalisme thermodynamique reliant la pression à la somme des orbites fermees du flot} and \ref{Lalley's theorem}. Finally, we will prove Theorem \ref{Sharp lambda} extending the Theorem of R. Sharp. The proof is very similar to the original one and we include it for the sake of completeness.

Here again $S_1$ and $S_2$ will be fixed hyperbolic surfaces.
\begin{definition}\label{def - correlation number}
For $\lambda \in  (\dil^-, \dil^+)$, the \emph{correlation number with slope $\lambda$ } is  
$$m(S_1,S_2,\lambda) := \lim_{T\tv \infty} \frac{1}{T} \log \Card \left\{ c \in \C , \ell_1(c) \in [T,T+1) \text{ and } \, \ell_2(c) \in [\lambda T , \lambda T +1) \right\}.  
$$
\end{definition}
The bound $\lambda T+1$ in the interval $ [\lambda T , \lambda T +1) $ has no consequences on $m$, we could have taken $\lambda T +k $ for any $k>0$.

Since $S_1$ and $S_2$ are fixed, we will note $m(\lambda)$  instead of $m(S_1,S_2,\lambda)$ for the correlation number of slope $\lambda$, cf Definition \ref{def - correlation number}  and $\delta (\lambda)$ instead of $\delta(S_1,S_2,\lambda)$ for the directional critical exponent, cf Definition \ref{def - critical exponent with slope}.

We first begin to prove an inequality : 
\begin{lemme}\label{inegaliteentreclambdaetdeltalambda}
$m(\lambda) \leq \delta(\lambda) (1+ \lambda) $.
\end{lemme}

\begin{proof}
Let  $\CC(T,\lambda) := \left\{ c \in \C , \ell_1(c) \in [T,T+1)  \text{ and } \, \ell_2(c) \in [\lambda T , \lambda T +1) \right\}$ the set of closed and "correlated" curves.  Let $\epsilon >0 $ and $c \in \CC(T,\lambda)$ then, for $T > \frac{\max(1,\lambda)}{\epsilon} $, $\left| \frac{\ell_2(c)}{\ell_1(c)} -\lambda \right|\leq \epsilon $, that is to say  $c\in \C(\lambda, \epsilon)$ (cf. eq (\ref{eq - def de CS(lambda epsilon)})).
$$\sum_{k \geq \max(1,\lambda)/\epsilon} \sum_{c \in \CC(k,\lambda)} e^{-s(\ell_1(c) +\ell_2(c))} \leq \sum_{c \in \C(\lambda, \epsilon)} e^{-s(\ell_1(c) +\ell_2(c))}.$$
The right hand side has critical exponent equal to $\delta(\lambda,\epsilon)$ and for $c\in\CC(T,\lambda)$ we have  $\ell_2(c)\leq \lambda T +1 \leq \lambda \ell_1(c) +1$, hence the left hand side satisfies

\begin{eqnarray*}
\sum_{k \geq \max(1,\lambda)/\epsilon} \sum_{c \in \CC(k,\lambda))} e^{-s(\ell_1(c) +\ell_2(c))}  &\geq & 
 \sum_{k \geq \max(1,\lambda)/\epsilon} \sum_{c \in \CC(k,\lambda)} e^{-s\ell_1(c) (1+\lambda) -s}.
\end{eqnarray*}
Since the growth of $\Card \CC(k,\lambda)$ is larger than $e^{(m(\lambda)-\eta) k}$ for all $\eta>0$, there is $k_0\in \N$   such that for every $k\geq k_0$
$$\Card \CC(k,\lambda)\geq  e^{(m(\lambda )-\eta) k}.$$
Set $k_1 := \max( k_0, \max(1,\lambda)/\epsilon ) $, 
\begin{eqnarray*}
 \sum_{k \geq \max(1,\lambda)/\epsilon} \sum_{c \in \CC(k,\lambda)} e^{-s(\ell_1(c) +\ell_2(c))} 
 &\geq & 
  \sum_{k \geq \max(1,\lambda)/\epsilon} \sum_{c \in \CC(k,\lambda)} e^{-s\ell_1(c) (1+\lambda) -s} \\
  &\geq & 
 \sum_{k \geq k_1 } \sum_{c \in \CC(k,\lambda))} e^{-s(k+1) (1+\lambda) -s} \\
 &\geq &
e^{-s(1+\lambda)-s}  \sum_{k \geq k_1}e^{(m(\lambda)-\eta ) k} e^{-s k (1+\lambda) }
\end{eqnarray*}
 
\textit{In fine}, the critical exponent of $\sum_{k \geq \max(1,\lambda)/\epsilon} \sum_{c \in \CC(k,\lambda)} e^{-s(\ell_1(c) +\ell_2(c))}  $ is larger  than $\frac{m(\lambda)-\eta}{1+\lambda}$, therefore we have for all $\epsilon>0$ and $\eta>0$, that 
 $m(\lambda) -\eta \leq \delta(\lambda, \epsilon) (1+\lambda)$. We conclude since $\eta$ and $\epsilon$ are arbitrary.
\end{proof}

Our proof of equality case use thermodynamic formalism that we survey in the next paragraph. 
\paragraph{Reparametrization of geodesic flow.} 

\begin{definition}
Let $\varphi_t$ be the geodesic flow on $T^1S_1$ and $\tau_1$ be a periodic orbit for $\varphi_t$. Let also $\psi : T^1 S_1 \tv \R$ be any Hölder continuous function. We note
$\omega(\psi ,\tau_1) : = \int_{\tau_1} \psi $ the integral of $\psi$ with respect to the arc length along the geodesic associated to $\tau_1$.
\end{definition}
  For example if  $\psi=1$, we have $\omega(\psi,\tau_1) = \ell_1(c)$, where $c$ is the closed geodesic on $S_1$ whose support is $\tau_1$.
We are going to construct a function $\psi$ such that $\omega(\psi ,\tau_1) =  \ell_2(c)$. This reparametrization is classic for example R. Schwartz and R. Sharp suggest a method to construct such a function.  Let's describe this construction. The classical references for what  we are going to introduce are  \cite{Ledrappier,Sambarino2,Sambarino1}. 

A Hölder cocycle is a map $u : \G\times \partial \G \tv \R$ satisfying :
$$u(\g' \g , \xi) = u(\g, \xi) + u(\g', \g \xi )$$
for every pair $\g,\g' \in \G$ and $\xi \in \partial\G$ , which  is a Hölder continuous map in the variable $\xi$.  Since the surface $S$ is compact the boundary of $\G$ can be identified with $\SU$ that is the boundary of $\Hyp^2$. The period of a cocycle $u$, $\ell_u(\g)$ is by definition $u(\g,\g^+)$ where $\g^+$ is the attracting fixed point of $\g$. 
From \cite[Corollary 1 p.106]{Ledrappier} for every Hölder cocycle $u$ there exists a function $\psi$ such that  $\omega(\psi ,\tau_1) =  \ell_u(\g)$, where $\g$ is the elements corresponding to the closed geodesic of support $\tau_1$. Now,  the following Busemann cocycle defined by  $\beta(\g,\xi ) := \lim_{x\tv \xi} d(o,x)-d(\rho_2(\g) o, x)$, where $o$ is any point of $\Hyp^2$, satisfies $\ell_\beta (\g) = \ell_2(\g)$. 
\begin{definition}\label{def - psi}
We call $\psi_{\rho_2}$ the (any) Hölder continuous function defined thanks to \cite[Corollary 1 p.106]{Ledrappier}  and the Busemann cocycle. 
\end{definition}
We then have 
$$\sum_{c\in \mathcal{C}} e^{-x\ell_1(c) -y\ell_2(c)} = \sum_{\tau} e^{\int_\tau -x -y\psi_{\rho_2}},$$
where the sum of the right hand side is taken over all closed orbit of $\varphi_t$. 

Given a Hölder continuous function $f : T^1 S_1 \tv \R$, we define the \textit{pressure} of 
$f$ to be 
$$P(f) = \sup_{\nu} \left\{ h(\nu) + \int_{T^1 S_1} f d\nu\right\},$$
where the supremum is taken over all ($\varphi_t$)-invariant  probability measures, and $h(\nu)$ is the entropy of the geodesic flow with respect to $\nu$.  This supremum is achieved by  a  unique such measure $\mu_f$, called the equilibrium state for $f$. In our setting of a surface of constant curvature,  the equilibrium state for $f\equiv 0$ is the Liouville measure, which is the local product of the Lebesgue measure for the metric associated to $S_1$ and the  arc length along the fibre.

We say that a function $v : T^1 S_1 \tv \R$ is continuously differentiable with respect to $\varphi_t$ if the limit 
$$v'(y) := \lim_{t\tv 0} \frac{v(\varphi_t (y)) - v(y)}{t}$$
exists everywhere, and is continuous. Two Hölder functions
$f$ and $g$ are said to be \textit{cohomologous} if $f-g= v'$ for some continuously differentiable function $v$. 

W. Parry and M. Pollicott showed in their book \cite[Proposition 4.8]{ParryPollicott}, that if $f$ is a Hölder continuous function which is not cohomologous to a constant, then the function $t\tv P(tf)$ is analytic, strictly convex and  furthermore, \cite[Proposition 4.12]{ParryPollicott}
$$P'(tf ):= \frac{d}{dt} P(tf) = \int_{T^1 S_1} \psi d\mu_{tf} $$
holds for each value of $t\in \R$. A version of this theorem more adapted to our notations can be found in \cite[p. 429]{SchwarzSharp}.

The following is a classical theorem in thermodynamic  theory, and has been proved by P. Walters in \cite[Theorem 4.1]{Walters} ; here again, the following  version  of this theorem in \cite[p. 106 eq. 10]{Ledrappier} is more adapted : 
\begin{theorem}\label{Formalisme thermodynamique reliant la pression à la somme des orbites fermees du flot}
Let $f : T^1S_1 \tv \R $ be a Hölder continuous function, then 
$$\lim_{T\tv +\infty} \frac{1}{T} \log \sum_{\tau , \, \omega({1,\tau})\leq T} e^{\int_\tau f} = P(f).$$
\end{theorem}

By the previous discussion the Manhattan curve is the set of $(x,y)\in \R^2$  such that \linebreak
$ \sum_{\tau} e^{-s\int_\tau x +y\psi_{\rho_2}}$ has critical exponent equal to 1.
Applying  \cite[Lemma 1, p106]{Ledrappier}  and the remark after it, to the function $f :=x+y\psi_{\rho_2}$, we see that critical exponent is equal to $1$ if and only if $P(-f)=0$ that is $P(-x-y\psi_{\rho_2})=0$. Hence 
$$\C_M= \{ (x,y) \, , \,  P(-x-y\psi_{\rho_2})=0\} =\{ (x,y) \, , \,  P(-y\psi_{\rho_2})=x\}  .$$

The \textit{independance Lemma} of \cite{SchwarzSharp} shows that $\psi_{\rho_2}$ is not cohomologous to a constant as soon as $\rho_1$ and $\rho_2$ are not conjugated. Hence we will thereafter suppose that $\rho_1$ and $\rho_2$ are not conjugated.   Moreover, the function $\psi_{\rho_2}$ satisfies the property that along any geodesic segment $s$ on $S_1$, $\int_s \psi_{\rho_2} = \ell_2(s)$, hence the function $\psi_{\rho_2}$ is strictly positive.  Taking the derivative of $t\tv P(-t\psi_{\rho_2})$ gives $\frac{\partial P(-t\psi_{\rho_2} )}{ \partial t}= - \int_{T^1 S_1} \psi_{\rho_2} d\mu_{-tf} \neq 0 $ and finally the implicit function Theorem assures the existence of  an analytic function $q(t)$ defined by  $P (-q(t) \psi_{\rho_2} ) =t$. By definition $\C_M$ is the graph of $q$.

We can then compute the derivative of $q$, as it is done in \cite{SchwarzSharp} : 
\begin{eqnarray*}
\frac{d}{dt} P(-q(t) \psi_{\rho_2}) = \left( - \int \psi_{\rho_2} d\mu_{-q(t)\psi_{\rho_2}}  \right) \frac{dq}{dt} = 1
\end{eqnarray*}
and we obtain 
\begin{eqnarray*}\label{formule pour dq/dt }
\frac{dq}{dt} = \frac{-1}{\int \psi_{\rho_2} d\mu_{-q(t)\psi_{\rho_2}} }.
\end{eqnarray*}

\begin{definition}
We set $J(f):= \{ P'(qf)\, , \, q \in \R \}$. By strict convexity , if $\lambda \in J(f) $, there is a unique  real number noted $q_\lambda$ such that $P'(q_\lambda f ) =\lambda$. 
\end{definition}

The next theorem is due to S.P Lalley \cite{Lalley}, 
\begin{theorem}\cite[Theorem 1]{Lalley}\cite[Proposition p.429]{SchwarzSharp}\label{Lalley's theorem} 
Let $f : T^1S_1 \tv \R$ be a Hölder continuous function, which is not cohomologous to a constant.  Let $\la\in J(f)$, there exists $K>0$ such that, 
$$\Card \{ \tau \,| \,  \omega(1,\tau) \in [T, T+1) \, \text{ and } \, \omega(f,\tau ) \in [\la T,\la T+1)\} \sim_{T\tv \infty} K\frac{e^{h(\mu_{q_\lambda f}) T}}{T^{3/2}}.$$ 
\end{theorem} 
Applying this theorem to the function $f=\psi_{\rho_2}$, defined in Definition \ref{def - psi}, which is not cohomologous to a constant, gives  for every  $\lambda \in J(\psi_{\rho_2})$,  the existence of $K>0$ such that, 
$$\Card \CC(T,\lambda) \sim_{x\tv \infty} K\frac{e^{h(\mu_{q_\lambda \psi_{\rho_2}})T}}{T^{3/2}},$$ 
hence if $\lambda \in J(\psi_{\rho_2}) $, we have $m(\lambda) = h(\mu_{q_\lambda \psi_{\rho_2}})$.

This next proposition does not figure  in the original paper of R. Sharp but the proof is essentially the same as in the remarks at the end of \cite{SchwarzSharp}
\begin{proposition}\label{les pentes sont incluses dans J(psi)}
For the function $\psi_{\rho_2}$ defined in Definition \ref{def - psi}, we have the following inclusion : $$(\dil^-, \dil^+ ) \subset J(\psi_{\rho_2}).$$
\end{proposition}

\begin{proof}
Let $\lambda \in (\dil^-,\dil^+)$. By definition there exists closed geodesics $c$ and $c'$ such that $\ell_2(c)< \la \ell_1(c)$ and  $\ell_2(c') > \la \ell_1 (c')$. They correspond to  closed orbits of $\varphi_t$, $\tau$ and $\tau'$, hence $\omega(\psi_{\rho_2}, \tau)/ \la ( 1, \tau) < \la$ et $\omega(\psi, \tau')/ \omega ( 1, \tau') > \la$. 

Let $I(\psi_{\rho_2})$ denote the set of values $\int \psi_{\rho_2} d\nu$  where $\nu$ ranges over invariant probability measures. Clearly $I(\psi_{\rho_2})$ is a closed interval.  If we take $\nu$ to be the probability measure whose support either $c$ or $c'$ we see that $\omega(\psi_{\rho_2}, \tau)/ \omega (1, \tau) \in I(\psi_{\rho_2})$, and $\omega(\psi_{\rho_2}, \tau')/ \omega(1, \tau') \in I(\psi_{\rho_2})$. Hence $]\la-2\epsilon, \lambda+2\epsilon[ \in I(\psi_{\rho_2})$ for some $\epsilon >0$. 

Since $\lambda \pm \epsilon \in I (\psi_{\rho_2})$, we have, 
$$\forall t \in \R, \quad \lambda t +\lvert t\lvert \epsilon \in t I (\psi_{\rho_2}).$$
And by definition of pressure if $y\in I(\psi_{\rho_2})$ then $P(t\psi_{\rho_2}) \geq ty$. Hence
$$\forall t \in \R , \quad P(t\psi_{\rho_2}) \geq \sup t I(\psi_{\rho_2}).$$
Combining the last two inequalities gives
$$\forall t\in \R, \quad P(t\psi_{\rho_2}) -\lambda t \geq \lvert t \lvert \epsilon. $$

Consider  the function $Q(t) := P(t\psi_{\rho_2}) -\lambda t $ ; we have $Q(0)=P(0)=1$ and we just proved that  for all $\lvert t\lvert >1/\epsilon$, $Q(t)>1$.  Therefore $Q$ has a minimum  $q_\lambda \in [-1/\epsilon, 1/\epsilon]$, where $Q'(q_\lambda) =0$ which is to say $P'(q_\lambda \psi) = \lambda$ and $\lambda\in J(\psi_{\rho_2})$. 
\end{proof}
Now we are going to prove the formula that extends the result of R. Sharp \cite{Sharp}.

\begin{theorem}\label{Sharp lambda}
Let $\lambda \in (\dil^-, \dil^+).$ There exists $K>0$ and  $m(\la)$ such that
$$\Card \,\CC(T,\lambda) \sim K \frac{e^{m(\la)T}}{T^{3/2}}.$$
Moreover  $m(\la) =x(\la) +\la y(\la)$.
\end{theorem}

\begin{proof}
Let $\lambda \in  (\dil^+, \dil^-),$ by Proposition \ref{les pentes sont incluses dans J(psi)},  we know that $\lambda \in J(\psi_{\rho_2})$. The remark after the Theorem \ref{Lalley's theorem} says that $m(\lambda)=h(\mu_{q_\lambda \psi_{\rho_2}})$. We have to show that $h(\mu_{q_\lambda \psi_{\rho_2}})= x(\la) +\la y(\la)$.   
By definition of $h(\mu_{q_\lambda \psi_{\rho_2}})$ we have that
$$m(\lambda)=h(\mu_{q_\lambda \psi_{\rho_2}})= P(q_\lambda \psi_{\rho_2})- \int q_\lambda\psi_{\rho_2} d\mu_{q_\lambda \psi_{\rho_2}}.$$ 
Set $x$ such that $q_\lambda = -q(x)$, this gives, 
$$m(\lambda) = P(-q(x) \psi_{\rho_2}) +q(x) \int \psi_{\rho_2} d\mu_{-q(x) \psi_{\rho_2}}= x+ q(x) \int \psi_{\rho_2} d\mu_{-q(x) \psi_{\rho_2}}.$$
But $q_\lambda$ is such that $\int \psi_{\rho_2} d\mu_{q_\lambda \psi_{\rho_2}} = P'(q_\lambda \psi_{\rho_2} ) = \lambda$, hence 
$$m(\lambda) = x+q(x)\lambda.$$
Moreover, $$\frac{dq}{dt} (x) = \frac{-1}{\int \psi_{\rho_2} d\mu_{-q(x) \psi_{\rho_2}} }  =\frac{-1}{\lambda}.$$ 
This is exactly where the slope of a  normal vector  at $q$ is equal to $\lambda$, that is $x=x(\lambda)$, and 
$$m(\lambda) = x(\lambda )+\lambda q(x(\lambda)) =x(\lambda )+\lambda y(\lambda).$$ 
\end{proof}

We finally conclude the proof of the equality case in Theorem \ref{rigidite de l exposant critique directionel}.
\begin{corollaire}\label{relatioC DELTA et courbe de manhattan}
Let $\lambda \in (\dil^-,\dil^+),$ we have 
$$\frac{m(\lambda)}{1+\lambda} = \delta(\lambda) = \frac{x(\lambda)+\lambda y(\lambda)}{1+\lambda}. $$ 
\end{corollaire}

\begin{proof}[Proof of Corollary \ref{relatioC DELTA et courbe de manhattan} and  Theorem  \ref{rigidite de l exposant critique directionel}]
From Lemma \ref{inegalitépourdeltalambda} and Lemma  \ref{inegaliteentreclambdaetdeltalambda}
we have for all $(x,y)\in \C_M$ and all $\lambda \in  (\dil^-, \dil^+)$  that 
$m(\lambda) \leq \delta(\lambda) (1+\lambda) \leq x+\lambda y$ and  with Theorem \ref{Sharp lambda}, equality occurs for $(x(\lambda),y(\lambda))$. 

We now show that the inequality $\delta(\lambda) \leq \frac{x+\lambda y}{1+\lambda}$ is strict for $(x,y)\neq (x(\lambda),y(\lambda) $. This  is a simple application of the strict convexity of the Manhattan curve. Let $x\tv q(x)$ be the function which graph is the Manhattan curve. Fix $\lambda$ and defined 
$$g(x) :=  \frac{x+\lambda q(x)}{1+\lambda}.$$  
The derivative of $g$ is $g'(x) =\frac{1+\lambda q'(x)}{1+\lambda}.$ By  Definition $\lambda = -1/q'(x_\lambda)$, hence
$g'(x)= \frac{q'(x(\lambda)) -q'(x)}{q'(x(\lambda)) - 1}$. By strict convexity of $q$, $g'(x)=0$ if and only if $x =x(\lambda)$, which finally implies that the inequality is strict for the others values of $(x,y)\in \C_M$.
\end{proof}

Here again, the following corollary  implicitly belongs to the work of  G. Link \cite{link2004hausdorff}. However, G. Link considers $\Produit$ endowed with the Riemannian metric, hence for the Manhattan metric this corollary was not known. 
\begin{corollaire}\label{l'exposant critique directionnel est maximum  quand il est egal a deltaM}\cite[compare to Theorem 3.12, Theorem 5.1]{link2004hausdorff}
For all  $\lambda\in (\dil^-, \dil^+)$ we have $$\delta(S_1,S_2,\lambda) \leq \delta(S_1,S_2),$$
and equality occurs for a unique $\lambda$ : the maximal slope (cf. Definition \ref{def maximal slope}).
\end{corollaire}

\begin{proof}
The inequality is obvious by definition. For the strict inequality we are going to use a similar method as the previous corollary. 
We use the formula  from  Corollary \ref{relatioC DELTA et courbe de manhattan}, $\delta(\lambda)=  \frac{x(\lambda)+\lambda y(\lambda)}{1+\lambda}.$ Recall that $\lambda(x)$ is  the slope of a normal vector at $(x,q(x))\in \C_M$. Hence $\delta(\lambda(x)) =  \frac{x+\lambda(x) q(x)}{1+\lambda(x)} =  \frac{xq'(x)- q(x)}{q'(x)-1}$, since $\lambda(x) = -1/q'(x)$. Denote by $h$ the function 
$$h(x)= \frac{xq'(x)- q(x)}{q'(x)-1}.$$ The derivative of $h$ is 
$h'(x) = \frac{q''(x) (-x +q(x))}{(q'(x)-1)^2},$
 which has the same the sign as $-x +q(x) $. By Proposition \ref{pr - basic properties}, $\delta$ is the intersection of $\C_M$ and the line $y=x$, this implies that the unique solution of $-x +q(x) =0$  is $x=q(x) =\delta$, and $h'(x)<0$ for all $x>\delta$ and  $h'(x)>0$ for all $x<\delta.$
Hence for all $x\neq \delta$ we have $\delta(\lambda(x)) = h(x)<h(\delta) = \delta$.
\end{proof}

\begin{corollaire}\label{delta(1) tends vers 0 si et seulement si delta(0) tends vers 0}
Let $S_n$ and $S_n'$ be two sequences  in the Teichmüller space  of $S$.  Then the following is equivalent : 
\begin{itemize}
\item $\lim_{n\tv \infty} \delta(S_n,S_n') =0 $
\item $\lim_{n\tv \infty} \delta(S_n,S_n',1) =0 $
\item  $\lim_{n\tv \infty} m(S_n,S_n',1) =0 $
\end{itemize}
\end{corollaire}

\begin{proof}
From Corollary \ref{relatioC DELTA et courbe de manhattan}, $\delta(S_n,S_n',1) = \frac{m(S_n,S_n',1)}{2}$ hence the second and third statements are clearly equivalent. We then show the equivalence of the first and the second ones. 

By definition $0\leq \delta(S_n,S_n',1) \leq \delta(S_n,S_n') $, hence  $\delta(S_n,S_n')  $ tends to zero, implies that $\delta(S_n,S_n',1) $ tends also to zero. 

Conversely, by Corollary \ref{relatioC DELTA et courbe de manhattan} $\delta(S_n,S_n',1) = \frac{a^n_1(1) +a^n_2(1)}{2}$, where $(a^n_1(1),a^n_2(1))$ is the point on the Manhattan curve $\C^n_M$ for the surfaces $S_n$ and $S'_n$, where the slope of one normal vector is $1$. From the convexity of $\C^n_M$ we have that $a^n_1(1)>0$ and $a^n_2(1)>0$. Hence  if $\lim_{n\tv \infty}\delta(S_n,S_n',1)= 0$ it follows that  $\lim_{n\tv \infty} a^n_1 (1)= 0$ and $\lim_{n\tv \infty} a^n_2(1) = 0$. By continuity of $\C^n_M$, this implies that there are points in $\C^n_M$ which are as close as we want from the origin, and since the critical exponent is equal to the abscissa of the intersection of $\C^n_M$ and the line $y=x$, this finally implies that the critical exponent can be made as small as we want. 
\end{proof}

\begin{corollaire}\label{delta(1) tends vers 1/2 si et seulement si delta tends vers 1/2}
Let $S_n$ and $S_n'$ be two sequences in the Teichmüller space  of $S$.  Then the following is equivalent : 
\begin{itemize}
\item $\lim_{n\tv \infty} \delta(S_n,S_n')=1/2 $
\item $\lim_{n\tv \infty}\delta(S_n,S_n',1) =1/2 $
\item  $\lim_{n\tv \infty} m(S_n,S_n',1) =1 $
\end{itemize}
\end{corollaire}

\begin{proof}
Here again the second and the third statement are  equivalent from the formula obtained in Corollary \ref{relatioC DELTA et courbe de manhattan}.

We now show the equivalence of the first two statements. By definition $\delta(S_n,S_n',1) \leq \delta(S_n,S_n')\leq 1/2$ hence the second statement implies the first. 

Conversely, if $\lim \delta(S_n,S_n',1) = 1/2$ by  Corollary \ref{Si delta tends vers 1 lambda tends vers 1}, $\lim \lambda(\delta(S_n,S_n',1))= 1$ 
since the the Manhattan curve is $C^1$, this implies that $\lim ({x^n(1) ,y^n(1)}) = \lim (\delta(S_n,S_n'),\delta(S_n,S_n')) =(1/2 , 1/2)$. Hence $\lim_{n\tv \infty} \delta(S_n,S_n',1) = \lim_{n\tv \infty}  \frac{x^n(1) +y^n(1)}{2} = \lim_{n\tv \infty}  \frac{\delta(S_n,S_n',1)+\delta(S_n,S_n',1)}{2} = 1/2$. 
\end{proof}

\section{Geometric background}\label{sec - geometric background}
\subsection{Geodesic currents}\label{sec:Geodesic currents}
The notion of geodesic currents has been introduced by F. Bonahon in \cite{BonahonGeometryofteichmuller}  in order to get,  in Bonhanon's words  " a better understanding of homotopy classes of unoriented closed curves on $S$".  But it is also a beautiful tool to understand  Thurston compactification of Teichmüller space by measured geodesic laminations, \cite{BonahonEarthquake}, or the ends of three manifolds as in \cite{BonahonBouts}. Geodesic currents are a generalisation of closed geodesics where the set of geodesic measured laminations as well as Teichmüller space $\Teich(S)$ have natural embeddings. All the material of this Section is well known and can be found in \cite{BonahonEarthquake,BonahonGeometryofteichmuller,BonahonBouts,Otal,Vsari}. We survey this material in order to introduce notations and will follow the lines of \cite{BonahonGeometryofteichmuller}.

Let $S$ be a surface of genus $g\geq2$, that we endow with a hyperbolic metric, and $\C$ the set of closed unoriented geodesics on $S$. In this Section, geodesics will be  primitive elements, ie represented by an indivisible element of $\G$. There is a bijective correspondence between the set of homotopy classes of unoriented closed curves and the set of unoriented closed geodesics with multiplicity. We didn't make this distinction in the previous Section since according to Knieper's theorem, the exponential growth of the number of primitive geodesics and non-primitive geodesics are the same, hence the resulting Manhattan curve and critical exponent are the same.

Consider $\widetilde{S}$ the universal covering of $S$, and $\GSt$ the set of geodesics of $\widetilde{S}$. Any closed geodesic lifts to a $\G$-invariant  set of $\GSt$. We can take the multiplicity into account if we identify the lifts with a Dirac measure on this discrete subset, where the Dirac measures are multiplied by the multiplicity of the geodesic. This is equivalent to put weights on geodesics. Hence a weighted sum of geodesics of $S$ is naturally a $\G$  invariant measure of $\GSt$. 

There is a parametrization of unoriented geodesics by the boundary at infinity of $\widetilde{S}$, say $\partial \widetilde{S}$, that is $\GSt \simeq \partial \widetilde{S} \times \partial \widetilde{S} - \Delta /\Z_2$, where $\Delta$ is the diagonal and $\Z_2$ acts by exchanging the two factors. Giving a hyperbolic metric on $S$ gives an identification between $\widetilde{S}$ and $\Hyp^2$, whose boundary is canonically $\SU$. Hence $\GSt$ is naturally identified with $\SU\times \SU -\Delta /\Z_2$. 

\begin{definition}
A geodesic current is a $\G$ invariant borelian measure on $\GSt$. The set $\GC$ of geodesic currents is endowed with the weak* topology defined by the family of semi-distances $d_f$ where $f$ ranges over all continuous functions $f : \GSt \tv \R$ with compact support and where $d_f(\alpha,\beta) = | \alpha(f) -\beta(f)| $.
\end{definition}
By the previous discussion the set of homotopy classes of closed curves on $S$ is embedded in $\GC$.

\begin{proposition}\cite[Proposition 2]{BonahonGeometryofteichmuller}
The space $\GC$ is complete and the real multiples of homotopy classes of closed curves are dense in it. 
\end{proposition}

\begin{definition}
If $[\alpha]$ and $[\beta]$ are two homotopy classes of closed curves on $S$, their geometric intersection is the infinimum of the cardinality of $\alpha' \cap \beta'$ for all $\alpha' \in [\alpha]$ and $\beta' \in [\beta]$. 
\end{definition}
This infinimum is realized for $\alpha'$ and $\beta'$  the geodesic representatives, \cite[Ch 3, Proposition 10]{FLP} . 

The crucial fact is that the geometric intersection can be extended to geodesic currents. 

\begin{theorem}\cite[Paragraph 4.2]{BonahonBouts}\cite[Proposition 3]{BonahonGeometryofteichmuller}\label{th - def intersection  et l'intersection est continue}
There is a continuous symmetric bilinear function $i : \GC\times \GC \tv \R^+$ such that, for any two homotopy classes of closed curves $\alpha,\beta \in \GC$, $i(\alpha,\beta)$ is the above geometric intersection. 
\end{theorem}

If we fix a current $\alpha$ whose support is sufficiently large (if it fills $S$), then the set of currents whose intersections with $\alpha$ is bounded, is compact. More precisely we have the following, 

\begin{theorem}\label{compacité des courants geodesic d intersection borné}\cite[Proposition 4]{BonahonGeometryofteichmuller}
Let $\alpha$ be a geodesic current with the property that every geodesics of  $\widetilde{S}$ transversely meets some geodesic which is in the support of $\alpha$. Then the set $\left\{ \beta \in \GC \, , \, i(\alpha,\beta) \leq 1 \right\} $ is compact in $\GC$. 
\end{theorem}

Taking a current which fills $S$ gives the following, 
\begin{corollaire}\cite[Corollary 5]{BonahonGeometryofteichmuller}
The space $\PGC$ of projective geodesic currents is compact. 
\end{corollaire}

We are not going to use exactly this fact but a related one : the set of geodesic currents of length less or equal to one is compact. This leads us to define a length for any geodesic current. 
Let $\F$ be the foliation of $T^1S$ by the orbit of geodesic flow. For each $\varphi_t$ invariant finite measure $\mu$  there exists an associated transverse measure to $\F$, $\widetilde{\mu}$,  normalized by the requirement that it is locally a product : 
$\mu =\widetilde{\mu}\times dt$, where $dt$ is the Lebesgue measure along orbits of $\varphi_t$ in $\F$. By definition $\widetilde{\mu}$ is a geodesic current.  Hence there is a bijective correspondence between the set of  invariant measures by the geodesic flow and the set of currents.  In particular the Liouville measure  gives rise to a geodesic current, $L$. This current depends on the hyperbolic metric and we will write the metric as a subscript if there is an ambiguity on the metric we consider on $S$. For example $L_m$ will be the Liouville geodesic current on $(S,m)$. 

\begin{theorem}\cite[Proposition 14]{BonahonGeometryofteichmuller}
Let $(S,m)$  be a hyperbolic surface. For every  closed curve $c \in \GC$, we have $i(L_m, c) =\ell_m(c)$. 
\end{theorem}
This remarkable property allows us to define the length of any geodesic current once we have set a hyperbolic metric on $S$, by the extension of $i$ to any geodesic current. Let $m$ be a hyperbolic metric, for every $\beta \in \GC$, by definition $\ell_m(\beta) := i(L_m,\beta)$

Since the Liouville current has the property to intersect transversely every  geodesics of $\widetilde{S}$ applying  Theorem \ref{compacité des courants geodesic d intersection borné} we get, 
\begin{theorem}
For any hyperbolic metric $m$, the set of geodesic currents of $m$-length equal to one is compact. That is the set $\{ \beta \in \GC\,  ,\,  i(\beta,L_m) =1 \}$ is compact. 
\end{theorem}

The Liouville currents are particularly interesting since they allow to embed Teichmüller space into the set of geodesic currents. 
\begin{theorem} \cite[Theorem 12]{BonahonGeometryofteichmuller}
The map $m \tv L_m$ is a proper embedding of Teichmüller space into the space of geodesic currents. 
\end{theorem}

\begin{definition}\label{def - intersection entre deux surfaces}
Let $S_1 = (S,m_1)$ and $S_2 = (S,m_2)$ be two surfaces endowed with hyperbolic metric.Then the \emph{intersection between $S_1$ and $S_2$ } is defined by
$$i(S_1,S_2) := i(L_{m_1}, L_{m_2}).$$
\end{definition}

For all $m\in \Teich(S)$  we have that $i(L_m,L_m)= \pi^2 |\chi(S)|$ \cite[Proposition 15]{BonahonGeometryofteichmuller} hence if $L_m =k L_{m'}$ then $k=1$ and $m=m'$. 
\begin{corollaire}\cite[Corollary 16]{BonahonGeometryofteichmuller}
The composition $\Teich(S) \tv \GC\tv \PGC$ is an embedding. 
\end{corollaire}
Finally here is the corollary we are going to use and which is clearly equivalent to the previous one, since there is a homeomorphism between the set of projective currents and the set of currents of length 1.  
\begin{corollaire}\label{compacite des courant geodesics pour une metrique fixée}
Let $(S_0,m_0)$ be a fixed hyperbolic surface, then the map $m\tv \frac{L_m}{i(m_0,L_m)}$ is an embedding of Teichmüller space into the space of geodesic currents of $S_0$ length 1. 
\end{corollaire}

%

\subsection{Geodesic laminations}
Let us introduce geodesic laminations in the context of geodesic currents. The simplest definition in this setting is,
\begin{definition}
A measured geodesic lamination $\Lam$  is a geodesic current whose self-intersection $i(\Lam,\Lam)$ is equal to $0$. The set of measured geodesic laminations will be denoted by $\MLam$.
\end{definition}

The original definition of measured laminations due to  W. Thurston \cite{Thurston1979geometry} has been shown to be equivalent to the above by F. Bonahon. 
Recall that according to Thurston, a measured geodesic lamination is a closed subset $\Lambda$ which is the union of simple disjoint geodesics endowed with a transverse measure, ie. a measure on every arc transverse to $\Lambda$ invariant by holonomy. The Thurston's topology on the set of measured laminations is defined by the semi-distances $d_\g (\alpha,\beta) =| i(\alpha,\g) - i(\beta,\g)|$ where $\g$ ranges over all simple closed curves on $S$. 
\begin{theorem}\cite[Proposition 17]{BonahonGeometryofteichmuller}
The set $\left\{ \alpha \in \GC \, , \, i(\alpha,\alpha)=0\right\} $ has a natural homeomorphic identification with Thurston's space of measured laminations. 
\end{theorem}

Let $\PMLam$ be  the set of projective measured laminations, that is the quotient of the space of geodesic measured laminations       by $\R^*_+$, where the equivalence is given by $(\Lam,\mu) \sim (\Lam, x\mu)$ for $x\in \R^*_+$. The topology of Thurston on $\Teich(S)\cup\PMLam$ is defined as follows: $\Teich(S)$ is open in $\Teich(S)\cup\PMLam$ and a sequence $m_j\in \Teich(S)$ converges to $\Lam$ if and only if the ratio $\frac{\ell_{m_j} (\alpha)}{\ell_{m_j} (\beta)}$ converges to $\frac{i(\Lam,\alpha)}{i(\Lam,\beta)}$  for every pair of simple closed curves $\alpha,\beta$ with $i(\Lam,\beta) \neq 0$. 
\begin{theorem}\cite[Theorem 18]{BonahonGeometryofteichmuller}
The topology of $\Teich(S)\cup\PMLam$ seen as a subspace of $\PGC$ is the same as the Thurston's topology. 
\end{theorem}
By Corollary \ref{compacite des courant geodesics pour une metrique fixée} if we fix a metric $m$,  the set of geodesic laminations of $m$-length 1 is compact. 

The simplest example of measured geodesic lamination is, of course, a weighted simple geodesic, or a disjoint union of such measures. Moreover, the next theorem, which is a version for laminations of Proposition 3.2, shows that these examples are dense. 
\begin{theorem}\cite[Proposition 4.9]{BonahonBouts} \cite{Thurston1979geometry}
$\MLam$ is the closure of the linear combinations of closed disjoint geodesics. 
\end{theorem}

\subsection{Earthquakes}
We introduced geodesic laminations in order to present a central tool in Teichmüller theory, the earthquake map. They are an extension of the notion of   Fenchel-Nielsen twists. Let $c$ be a closed simple geodesic on $(S,m)$, take a tubular neighborhood $U\simeq \SU\times [-\varepsilon,\varepsilon]$ of $c$, such that $\SU\times\{0\}$ is isometrically sent to $c$, and every $z\times[-\varepsilon,\varepsilon]$ corresponds to a geodesic arc orthogonal to $c$. Consider a smooth function $\xi : [0,\varepsilon] \tv \R$ equal to $0$ on a neighborhood of $\varepsilon$ and equal to $1$ on a neighborhood of $0$. Then, for $t\in \R$  consider the map $\phi_t : S \tv S$ which is the identity on $S-U$ and on $\SU\times [-\varepsilon,0[ \subset U$ and which is defined by $\phi_t(e^{i\theta},u) =\left( e^{i\theta-i t\xi(u)},u\right)$ on $\SU \times [0,\varepsilon] \subset U$. The map $\phi_t$ is a diffeomorphism on $S-c$, discontinuous along $c$. The metric $\phi_t^* (m)$ on $S-c$ coincides with the original metric $m$ if we are sufficiently close to $c$, and it extends to a hyperbolic metric over all $S$. We denote the new metric by $\EQ{c}{t}{m}$, which is, in Thurston terminology, the metric obtained from $m$ by a left earthquake of amplitude $t$ along the curve $c$. 
When $c$ is a simple closed curve, this transformation is more classically called Fenchel-Nielsen twist, the extension to any geodesic lamination has been studied by W. Thurston, W. Kerckhoff, and F. Bonahon among others. 

\begin{theorem}\cite{kerckhofflinesofminima}\cite[Proposition 1]{BonahonEarthquake}\label{continuité des tremblement de terre}
There is a continuous function $\mathcal{E} : \MLam \times \R \times \Teich(S) \tv \Teich(S)$, associating to $m\in \Teich(S)$ an element $\EQ{\Lam}{t}{m} \in \Teich(S)$ such that $\EQ{\lambda \Lam}{t}{m}= \EQ{\Lam}{\lambda t}{m}$, for all $\lambda>0$ and all $\Lam \in \MLam$, and coincide with a Fenchel-Nielsen twist when $\Lam$ is a closed geodesic.   
\end{theorem}

Moreover S. Kerckhoff showed that the length function is convex along earthquake. 
\begin{theorem}\label{Convexity along earthquake}\cite[p253, Theorem 1]{kerckhoff1983nielsen}
Let $m$ be a hyperbolic metric, $\Lam$ be  a  measured geodesic lamination, and $c$ be a closed curve. Then the function $t\tv \ell(\EQ{\Lam}{t}{m} c) = i (\EQ{\Lam}{t}{m} , c)$ is convex.
\end{theorem}

The last tool we are going to use is the "geology" Theorem of Thurston saying that any two hyperbolic metrics are linked by an earthquake. 
\begin{theorem}\cite[Appendix]{kerckhoff1983nielsen}
Let $(S_1,m_1)$ and $(S_2,m_2)$ be  two hyperbolic metrics on $S$, then there exists a unique measured lamination $\Lam$ such that 
$$\EQ{\Lam}{1}{m_1} =m_2.$$
\end{theorem}
We are going to use the equivalent following corollary, which is just a renormalization of the measured lamination seen as current. 
\begin{corollaire}\label{Geology}
Let $(S_1,m_1)$ and $(S_2,m_2)$ be two hyperbolic metrics on $S$, then there exists a unique measured lamination $\Lam$ of $m_1$-length $1$ and a $T \in \R^+$ such that :
$$\EQ{\Lam}{T}{m_1} = m_2.$$
\end{corollaire}

\subsection{Examples}\label{subsec - Examples}
In the paper \cite{Sharp}, R. Sharp is asking about the behaviour of the correlation number $m(S_1,S_2,1)$ as $S_1$ and $S_2$ range over $\Teich(S)$. 
By Corollaries \ref{delta(1) tends vers 0 si et seulement si delta(0) tends vers 0} and \ref{delta(1) tends vers 1/2 si et seulement si delta tends vers 1/2}, the asymptotic behavior of $m(S_1,S_2,1)$, $\delta(S_1,S_2)$ and $\delta(S_1,S_2,1)$ are the same. Hence our examples answer his question in a quantitative way.

Moreover, this section establishes the counterpart of McMullen's examples for globally hyperbolic $\AdS$ manifolds. 
\paragraph{Proof of Examples}
The main ingredients for the different proofs are Proposition \ref{pr - quand un courant remplit on peut comparer à une fonction longueur}, and the convexity of the length function along earthquake paths, Theorem \ref{Convexity along earthquake}.
Example 1 is proved in Proposition \ref{pr - example pseudo diff delta tends vers 0}, we use the fact that a pseudo-Anosov diffeomorphism dilates the length of the curves  along one lamination, and its inverse dilates along another lamination. Those two laminations fill the surface, therefore we conclude with Proposition \ref{pr - quand un courant remplit on peut comparer à une fonction longueur}. \\
Example 2 is proved in \ref{sssection pinching one geodesic}. It is the more straightforward. Since there exists some curves for which the length is bounded in both surfaces, this immediately gives the bound on the critical exponent. \\
Example 3 is particularly interesting for our work, since the proof gives the idea for Theorem \ref{th - isolation 1 surface}. It uses the convexity of length function along earthquakes and invariance of critical exponent through the diagonal action of the mapping class group, see Section \ref{sssection Dehn twists}. \\
Example 4 follows is proved in \ref{sssection - fenchel nielsen twist} from the fact that the critical exponent is uniformly continuous, Proposition \ref{pr - delta is uniformly continuous}. The proof of this last fact is not so direct, it uses the machinery of geodesic currents and earthquakes from the previous sections.\\
Finally we prove also the counter example of Corollary \ref{cor - isolation 2 surfaces}, when the surfaces do no stay in the thick part of the Teichmüller space, Theorem \ref{L'exposant critique tends vers 1/2 quand on pince une geodesique}. This is done in \ref{sssection  - pinching at different speed}. The main argument used is the fact that the Thurston and the Weil-Petersson distances are not comparable, more precisely that the Weil-Petersson distance is not complete. 

\subsubsection{Anosov diffeomorphisms}\label{ssection anosov diff}
In this first example, we iterate a pseudo-Anosov diffeomorphism to show that the critical exponent can tend to $0$. 
\begin{definition}\cite[Section 13]{cannon2007group}
A diffeomorphism $A : S \tv S$  is said \emph{pseudo-Anosov} if there exists two measured geodesic laminations $(\Lam_+,\mu_+)$, $(\Lam_-,\mu_-)$, one transverses to the others, and a constant $k>1$ such that : 
\begin{enumerate}
\item $\Lam_+\cup\Lam_-$ fills $S$. 
\item $A (\Lam_\pm) = \Lam_\pm.$
\item $A^*\mu_+ = k \mu_+.$ (inverse image of $\mu_+$ by $A$)
\item $A^*\mu_- = \frac{1}{k} \mu_-.$ (inverse image  of $\mu_-$ by $A$)
\end{enumerate}
The laminations $\Lam_+, \Lam_-$ are called attracting and repelling, the constant $k$ is called the dilatation of $A$. 
\end{definition}
Let  $(S,m_0)$ be a hyperbolic surface. 
Let $A$ be a pseudo-Anosov  diffeomorphism and $\Lam_\pm$ its associated laminations, $k$ its dilatation.  We consider the following sequences of hyperbolic surfaces :
$$S_n :=(S, {A^*}^n(m_0))$$
et 
$$S'_n :=(S, {A^*}^{-n}(m_0))$$
\begin{proposition}\label{pr - quand un courant remplit on peut comparer à une fonction longueur}
Let  $\alpha$ be a geodesic current which fills $S$ and $\beta$ any geodesic current. There exists $K>0$, such that for every $c\in \C$ we have : 
$$i(\alpha,c)\geq \frac{1}{K} i(\beta, c).$$
\end{proposition}
\begin{proof}
Indeed, the  function $c\tv i(\beta,c)$ is continuous on the set of geodesic currents. As $\{c\in \GC \, | \, i(\alpha,c)= 1\}$ is compact, by  Theorem \ref{compacité des courants geodesic d intersection borné}, there exists $K>0$ such that for every $c\in \{c\in \GC \, | \, i(\alpha,c)= 1\}$  we have
$$ i(\beta,c) \leq K.$$
We conclude by homogeneity of this last formula. 
\end{proof}

In particular, using the preceding proposition with $\alpha = L_0$, the Liouville current associated to the metric $m_0$ and $\beta = \Lam_\pm$ we have  
\begin{corollaire}\label{cor - quand un courant remplit on peut comparer à une fonction longueur}
There exists $K>0$ such that for every $c\in \C$ we have
$$\ell_{S_0}(c)\geq \frac{1}{K} i(\Lam_\pm, c).$$
\end{corollaire}

We also need the following property of the mapping class group\footnote{(Mapping class group is defined by $MCG = Diff(S)/Diff_0(S) $, where $Diff_0(S)$ is the group of diffeomorphisms isotopic to the identity)}. This shows that they act on intersection forms as "isometries".
\begin{proposition}\label{pr - un diffeo agit par isométrie sur l'intersection}
Let $D$ be  a diffeomorphism of $S$. For every pairs of geodesic current $(\alpha,\beta) \in \GC$ we have
$$i(D\alpha,D\beta) = i(\alpha,\beta).$$
\end{proposition}
\begin{proof}
These is true for every currents which are union of closed curves, since $D$ is a diffeomorphism and does not change the geometric interesction. We conclude by density and continuity of the intersection. 
\end{proof}

We finally gives the bounds on the length of a geodesic on  $S_n$ and $S_n'$. 
\begin{proposition}\label{pr inegalite}
There exists $K>0$ such that for every $c\in \C$ and every $n\in \N$ we have 
$$\ell_n(c)\geq \frac{k^n}{K} i(\Lam_+, c).$$
$$\ell_n'(c)\geq \frac{k^n}{K} i(\Lam_-, c).$$
\end{proposition}

\begin{proof}

Indeed, we just saw that $ \ell_n(c) = i(A^{-n} L_0 ,c) = i(L_0,A^n c)$. Hence, from Corollary \ref{cor - quand un courant remplit on peut comparer à une fonction longueur}, there exists $K>0$ such that $$\ell_n(c)\geq \frac{1}{K} i(\Lam_+, A^n c).$$ 
Thanks to the Proposition \ref{pr - un diffeo agit par isométrie sur l'intersection}, we have $$i(\Lam_+,A^n c) = i(A^{-n} \Lam_+, c).$$ The third property of pseudo-Anosov diffeomorphisms implies that  $$  i(A^{-n} \Lam_+, c)= k^n i(\Lam_+, c).$$ 
These three relations shows the first inequality of the Proposition \ref{pr inegalite}. The second inequality can be shown by the same method. 
\end{proof}

\begin{proposition}\label{pr - example pseudo diff delta tends vers 0}
Let $A$ be a pseudo-Anosov diffeomorphism, then 
$$\lim_{n\tv \infty}\delta_{Lor}(S_0,A^{2n}S_0) = 0.$$
\end{proposition}
\begin{proof}
We just showed there exists $K>0$ such that for every $c\in \C$
\begin{eqnarray*}
\ell_n(c) +\ell'_n(c) &\geq &\frac{ k^n}{K} i(\Lam_+ \cup \Lam_- ,c) 
\end{eqnarray*}

As $\Lam_+ \cup \Lam_- $ fills $S$,  Proposition \ref{pr - quand un courant remplit on peut comparer à une fonction longueur} shows there exists $K'>0$ such that for every $c\in \C$ we have
$$i(\Lam_+ \cup \Lam_- ,c) \geq \frac{1}{K'} i(L_0,c) =\frac{\ell_0(c)}{K'}.$$

Hence, we have

\begin{eqnarray*}
\ell_n(c) +\ell'_n(c) &\geq & \frac{k^n}{K K'} \ell_0(c).
\end{eqnarray*}

Finally, there exists $K''>0$ such that 
$$\sum_{c\in \C} e^{-s(\ell_n(c) +\ell'_n(c) )} \leq \sum_{c\in \C} e^{-s k^n K'' \ell_0(c)} $$
Hence,  critical exponent associated to $(S_n,S_n')$, satisfies $\delta(S_n,S_n') \leq \frac{1}{k^n K''}$ and tends to  $0$. 

Finally, remark that we can fix $S_0$. Indeed, let $D$ be an element of the  mapping class group. The critical exponent $\delta(S,S')$ is equal to $\delta(D(S),D(S'))$, since the length spectrum is invariant by $MCG$. Let  $S_0$ be a hyperbolic surface and consider the critical exponent associated to $(S_0,A^{2n}S_0)$. This exponent is then equal to $\delta(A^{-n} S_0, A^{n}S_0)$, which tends to $0$ from what we just saw. 
\end{proof}

\subsubsection{Shrinking transverse pair of pants}\label{sssection pinching transverse pair of pants}
 Let $S_0$ be a hyperbolic surface. Take two pants decomposition of $S_0$, $\Pants$ and  $\Pants'$, such that $\Pants \cup \Pants'$ fills up $S_0$ (ie the complement of $\Pants \cup \Pants'$ consists of topological disks).
 \begin{definition}
A decomposition of a surface $S$ by pair of pants $\Pants$ give Fenchel-Nielsen coordinates, "length" and "twist" coordinates.  A surface is said to be \emph{shrinked} along a simple  geodesic $c\in \Pants$ if the length coordinates of $c$ tends to $0$ and the others do not change. It is said \emph{shrinked by a factor $x$} if the length coordinates of $c\in \Pants$ is replace by $x\ell(c)$.
 \end{definition}
 Shrinking is defined via the hyperbolic structure of $S$. In Section \ref{sssection  - pinching at different speed}, we will call \emph{pinching} the equivalent procedure defined via the complex structure as defined in \cite{Wolpert}. The important fact is the existence of a very small geodesic after shrinking/pinching. \\
 We consider the surfaces $S_n$ and $S_n'$, defined by shrinking the geodesics of $\Pants$ respectively $\Pants'$ by a factor $e^{-n}$.  By the \cite[collar Lemma]{halpern1981proof}, there exists $M$ such that for every $c \in \C, \, \ell_n(c) \geq C | \log(e^{-n})| i(c,\Pants) = C i(c,\Pants) n $, since the length of the geodesics in $\Pants$ are less than $e^{-n}$. We also have  $\ell'_n(c) \geq M | \log(e^{-n})| i(c,\Pants')=n {M i(c,\Pants')}$. These two inequalities give, 
$$\ell_n(c) +\ell'_n(c) \geq n M i(c ,\Pants \cup \Pants'),$$
By Proposition \ref{pr - quand un courant remplit on peut comparer à une fonction longueur} and since $\Pants \cup \Pants'$ fills up $S_0$, there exists $K_0>0$ such that for all $c\in \C$ we have $ i(c ,\Pants \cup \Pants') \geq K_0 \ell_0(c)$, hence
$$\ell_n(c) +\ell'_n(c) \geq n M K_0 \ell_0(c).$$
Again $\delta(S_n,S_n') \leq \frac{1}{n M K_0}$ goes to zero.

\subsubsection{Dehn twists}\label{sssection Dehn twists} The next example contains the basic idea for the proof of Theorem \ref{th - isolation 1 surface} :  we show that along a sequence obtained by iteration of a  Dehn twist, the critical exponent is decreasing.  Let $S_0:=(S,m_0)$ be a hyperbolic surface and $\alpha $ a simple closed curve on $S$. Let $t_\alpha$ be the Dehn twist along $\alpha$ and define $S_n$ (respectively  $S'_n$) by $S_n := t_\alpha ^n S_0 $ ( respectively $S_n':=t_\alpha ^{-n} S_0 $). By definition  we have $S_n =  \EQ{\alpha}{n \ell_0(\alpha)}{S_0}$. We denote for $t\in \R$ the surface $S_t :=  \EQ{\alpha}{t \ell_0(\alpha)}{S_0}$. Let $c$ be a closed curve on $S$  and $f$  be the function $f \, : \, t \tv \ell_t (c)$, that is the length of the geodesic representative of $c$ on $S_t$. By Theorem \ref{Convexity along earthquake} the function $f$ is convex. Finally let $g(t) = f(t) +f(-t)$ and remark that 
\begin{eqnarray}
g(n) &=& f(n)+f(-n)\\
		&=&\ell_n(c) +\ell'_n(c)
\end{eqnarray}

$g$ is convex and $g'(0) =0$ hence, $g$ is an increasing function on $\R^+$. 
Hence $g(n)\geq g(n-1)$ or equivalently,
$$\ell_n(c) +\ell'_n(c)\geq \ell_{n-1}(c) +\ell'_{n-1}(c)$$

This implies that $n\tv \delta(S_n,S'_n)$ is decreasing. Moreover, by Theorem  \ref{th - rigidité des GHMC}, $\delta(S_1,S'_1)<1/2$ and we can then conclude 
$$\lim\delta(S_n,S'_n) < 1/2. $$
The limit exists since $\delta(S_n,S'_n)$ is decreasing. 

As in the first example, this shows that $\lim\delta(S_0,\tau_\alpha^{2n}S_0) < 1/2.$ This result is the one we want to generalise thanks to earthquakes.

\subsubsection{Shrinking one geodesic}\label{sssection pinching one geodesic} 
Let $\Pants$ be a pants decomposition of a hyperbolic surface $S_0$. Let $\alpha_i$, $i\in [1,3g-3]$ be the geodesics boundaries of $\Pants$. We call $S_t$ the surface shrinked along $\alpha_1$ by a factor $e^{-t}$. This means that on $S_t$ the length of $\alpha_1$ is $\ell_t(\alpha_1)=e^{-t}\ell_0(\alpha_1)$, and the length of $\alpha_i$, $i\in [2,3g-3]$ is $\ell_t(\alpha_i)=\ell_0(\alpha_i)$. \\
Let $\rho_t$ be the sequence  of representations of $\G$ in $\PSL_2(\R)$ such that $\Hyp^2/\rho_t(\G) =S_t$. 
Let $g_i\in \G$, $i \in [1,3g-3]$ be elements of the fundamental group projecting to $\alpha_i$, $i \in [1,3g-3]$. We can choose $\rho_t$ in order that $\rho_t(\alpha_i)$ is fixed for all $i \geq 2$. By definition the critical exponent associated to $(S_0,S_t)$ is larger than the critical exponent associated to $(S_0\setminus \{\alpha_1\}, S_t\setminus \{\alpha_1\})$, in which we do not count the curves meeting $\alpha_1$. \\
Moreover, the critical exponent associated to $(S_0\setminus \{\alpha_1\}, S_t\setminus \{\alpha_1\})$ is constant and positive. This shows that 
$$\liminf_{t\tv \infty} \delta(S_0,S_t) >0.$$
This example as to be compared to the equivalent  of \cite[p.3  Example 3]{mcmullen1999hausdorff}.


\subsubsection{Fenchel-Nielsen twists}\label{sssection - fenchel nielsen twist}
We show  there is  a periodic limiting function for critical exponent along Fenchel-Nielsen twist.
 Let $\alpha$ be a simple closed curve and $\mathcal{E}_{\alpha}^{t}$ the Fenchel-Nielsen twist along $\alpha$, recall that $\tau :=\mathcal{E}_{\alpha}^{\ell_0(\alpha)}$ is the Dehn twist along $\alpha$. Let $S_0$ be a hyperbolic surface and defined $S_t := \EQ{\alpha}{t}{S_0}$. Fix $t \in [0,2\ell_0(\alpha))$, then as the previous example shows we see that $\delta(S_0,S_{2\ell_0(\alpha)n+t} )=\delta(\tau^{-n}S_0,\tau^n S_{t} ) $ is decreasing, hence we can consider the following function : 
$$\delta(t) := \lim_{n\tv \infty} \delta(S_0,S_{2\ell_0(\alpha)n+t} ).$$
Obviously $\delta$ is $2\ell_0(\alpha) $-periodic. 

\begin{theorem}\label{th - periodic fenchel nielsen}
Let $S_t$ be a surfaces path obtained by Fenchel-Nielsen twist along a simple closed geodesic $\alpha$. Then there exists a $2\ell_{S_0}(\alpha)$-periodic  function $\delta $ such that: 
$$\lim_{t\tv \infty} | \delta(S_0,S_t) -\delta(t) | =0.$$
\end{theorem} 
We need the following result of real analysis 
\begin{lemme}
Let $f\,: \R\tv \R$ be a continuous real function such that $F(t):=\lim_{n\tv \infty} f(t+n)$ exists for all $t\in [0,1)$. If $f$ is uniformly continuous then $F$ is continuous and 
$$\lim_{t\tv \infty} | F(t) - f(t) | =0.$$
\end{lemme}

\begin{proof}
Let $\epsilon>0$. Let $\delta>0$ such that for all $t\in \R$, all $t'\in (t-\delta,t+\delta),$ we have $| f(t')-f(t)| \leq \epsilon.$ We then have for all $n\in \N$, for all $t,t'\in\R$ such that $|t-t'| \leq \delta$ 
$$|f(t+n) - f(t'+n)|\leq \epsilon.$$
Taking the limit in $n$ shows that $|F(t)-F(t')|\leq \epsilon,$ and the continuity of $F$ follows. 

Now for all $t\in [0,1]$, there exists $N>0$ such that for all $n>N$ we have:
$$| f(t+n) -F(t)|\leq \epsilon.$$
Since $F$ is $1$-periodic, we have: 
$$| f(t+n) -F(t+n)|\leq \epsilon.$$
From uniform continuity, for all $t'\in (t-\delta, t+\delta)$ we have
$$| f(t'+n)-F(t'+n)| \leq 3\epsilon.$$
Covering $[0,1]$ by a finite number of segments of the form $(t-\delta,t+\delta)$, shows there is $N\in \N$ such that for all $n>N$ and all $t\in [0,1]$ we have:
$$| f(t+n)-F(t+n)| \leq 3\epsilon.$$
In other words, there is $N>0$ such that for all $t\geq N$ we have:
$$|f(t)-F(t)|\leq 3 \epsilon.$$
This shows the second part of the proposition. 

\end{proof}

The following proposition will finish the proof Theorem \ref{th - periodic fenchel nielsen}
\begin{proposition}\label{pr - delta is uniformly continuous}
 The function $t\mapsto \delta(S_0,S_t)$ is uniformly continuous. 
\end{proposition}

\begin{proof}
We define the following function 
$$g(t) :=\sup_{c\in \C} \frac{i(c,\alpha)}{\ell_{S_t}(c)}.$$
This function is well defined by Proposition \ref{pr - quand un courant remplit on peut comparer à une fonction longueur}. By homogeneity of  $ \frac{i(c,\alpha)}{\ell_{S_t}(c)} $, we could have taken for definition the supremum on geodesic currents of $S_0$-length equal 1. This set is compact, hence the function $g$ is continuous. Moreover it is  $\ell_0(\alpha) $-periodic :
\begin{eqnarray*}
g(t+\ell_0(\alpha)) &=& \sup_{c\in \C} \frac{i(c,\alpha)}{\ell_{S_{t+\ell_0(\alpha)}}(c)}.
\end{eqnarray*}
\text{Since}  $\tau =\EQ{\alpha}{\ell_0(\alpha)}{ }$ \text{we have  }
\begin{eqnarray*}
g(t+\ell_0(\alpha)) &=&  \sup_{c\in \C} \frac{i(c,\alpha)}{\ell_{S_t}(\tau (c))}.
\end{eqnarray*}
As $i$ is invariant by the diagonal action of the mapping class group \ref{pr - un diffeo agit par isométrie sur l'intersection}, 
\begin{eqnarray*}
g(t+\ell_0(\alpha)) &=&  \sup_{c\in \C} \frac{i(\tau c,\tau \alpha)}{\ell_{S_t}(\tau (c))}.
\end{eqnarray*}
And $\alpha$ is invariant by $\tau$, hence : 
\begin{eqnarray*}
g(t+\ell_0(\alpha)) &=&  \sup_{c\in \C} \frac{i(\tau c, \alpha)}{\ell_{S_t}(\tau (c))}.
\end{eqnarray*}
We conclude by a change of variable $\tau c =c'$
\begin{eqnarray*}
g(t+\ell_0(\alpha)) &=&  \sup_{c'\in \C} \frac{i(c', \alpha)}{\ell_{S_t}(c')}\\
							&=& g(t).
\end{eqnarray*}
Hence $g$ is bounded on $\R$, we call $M$ its maximum. 

From convexity of earthquake paths, we have  for all $t\in \R, \varepsilon>0$ and $c\in \C$
$$\ell_{S_t} (c) - \varepsilon i(c,\alpha)  \leq \ell_{S_{t+\varepsilon}} (c) \leq \ell_{S_t} (c) +\varepsilon i(c,\alpha).$$
 We then have 
$$\ell_{S_t} (c) - \varepsilon M \ell_{S_t} (c)   \leq \ell_{S_{t+\varepsilon}} (c) \leq \ell_{S_t} (c) +\varepsilon M \ell_{S_t} (c).$$
Therefore: 
$$(\ell_{S_t}(c) +\ell_{S_0}(c))(1-\varepsilon M)  \leq  \ell_{S_{t+\varepsilon}} (c)  +\ell_{S_0}(c) \leq (\ell_{S_t}(c) +\ell_{S_0}(c))(1+\varepsilon M), $$

$$1-M\varepsilon \leq \frac{\ell_{S_{t+\varepsilon}} (c)+\ell_{S_0}(c)  }{\ell_{S_{t}} (c)+\ell_{S_0}(c) } \leq 1+ M\varepsilon.$$
Passing to critical exponent we finally obtain that for all $\eta>0$ there exists $\varepsilon>0$ such that for all $t\in \R$. 
$$| \delta(S_0,S_{t+\varepsilon} ) -\delta(S_0,S_{t}) | \leq \eta.$$
\end{proof}

This is the result equivalent to \cite[Theorem 9.6]{mcmullen1999hausdorff} for GHMC manifolds. Whereas C. McMullen said that the function $t\tv \delta(t)$ should not be constant in the quasi-Fuchsian case, it seems legit to think in our context that $t\tv\delta(t)$ is constant over $\R$. The reason for this is we have a natural candidate for the limit of $\delta(S_0,S_t)$. Indeed, the length of all the curves crossing $\alpha$  growth to infinity (but \emph{not} uniformly) hence we can conjecture that $t\tv\delta$ is constant and equal to the  abscissa   of convergence of 
$$\sum_{c\in \C, \, i(c,\alpha) =0} e^{-s2\ell_{S_0} (c)}.$$

\subsubsection{Pinching at different speed}\label{sssection  - pinching at different speed} Finally we give an example of a family of surfaces $(S_t, S_{t'})$  for which the Thurston distance between the two representations is bounded below (or even tends to infinity), but the critical exponent tends to $1/2$. Let $S_0$ be  a hyperbolic surface, and $c$ be a simple closed curve. Let $S_t$ the hyperbolic surface obtained by pinching $c$ as it is explained in Wolpert article \cite{Wolpert}. We don't want to explain this construction since it uses Jenkins-Strebel  differential, which are far from the subject of the present paper. The two things we have to know for our example is, first that the Weil-Petterson lenght of the path $t\tv S_t$ is finite, and second that the length of the geodesic $c$ converges to $0$. 
\begin{theorem}\label{L'exposant critique tends vers 1/2 quand on pince une geodesique}
There is a sequence $t_n$ such that the family $d_{Th}(S_{n},S_{t_n}) \tv +\infty$, and such that $\lim_{n\tv \infty} \delta(S_{n},S_{t_n}) = 1/2$.
\end{theorem}

\begin{proof}
The length of $c$ tends to 0, therefore, for all $n$ there is $t_n>n$ such that 
$\ell_{S_n}(c) \geq \ell_{S_{t_n}}(c) e^n$. In this way, the Thurston symmetric distance between $(S_{n},S_{t_n})$ is bigger than $\log (\frac{\ell_{S_n}(c) }{\ell_{S_{t_n}}(c) }\geq n$. 

The result on the critical exponent is due to two facts. The first one is, as we already said, that the Weil-Petersson distance of the path $p : \R^+ \tv \Teich(S)$, $p(t) = S_t$ is finite, this is proved by Wolpert in \cite{Wolpert}. The second fact is the link between the intersection and Weil-Petersson metric. Bonahon showed in \cite[Theorem 19]{BonahonGeometryofteichmuller} that 
$i(m,m') = i(m,m) +o(d_{WP} (m,m'))$.

The first fact assures that $d_{WP}(S_n,S_{t_n})  \tv 0$ as $n\tv \infty$. With the second fact, we can conclude that $i(m_n,m_{t_n}) \tv i(m_n,m_n)$. Now by Corollary \ref{Si I tends vers 1 delta tends vers 1/2} we can conclude that $\lim_{n\tv \infty}\delta(S_n,S_{t_n}) = 1/2$.
\end{proof}

%

\subsection{Large deviation theorem}\label{subsec -large deviation theorem}
Finally we will use a Theorem of large deviation for orbits of geodesic flow due to Y. Kifer \cite{kifer} (in a much more general context). Let $\Pants$  be the set of $\varphi_t$-invariant  probability on $T^1S$. From a closed geodesic $c$ we can defined a borelian, invariant measure $\hat{c}$ as follow. Let $E\subset T^1S_1$ any borelian, $\chi$ the characteristic function of $E$ and $v$ any vector of $\dot{c}$ :
$$\hat{c} (E) = \int_0^{l(c)} \chi_E(\varphi_t(v)) dt.$$
Every closed geodesic $c$ can be considered as an element of $\Pants$, through $\frac{\hat{c}}{l(c)}$.
 By the discussion in Section \ref{sec:Geodesic currents}, the set of invariant probability measures is in bijection with the set of geodesic currents of length equal to 1. Let $\widehat{L}$ be the Liouville measure on $T^1S$ (which corresponds to the Liouville current $L$). We give a name for the set of geodesics of $S$ of length less than $T$ :
\begin{equation}\label{eq - def de CS(T)}
\C(T) :=\left\{ c\in \C \, |\, \ell(c) \leq T\right\}.
\end{equation}
 
\begin{theorem}\label{Large deviation au niveau des groupes}
For any open neighborhood  $\mathcal{U}$ of $\widehat{L}$ in $\Pants$, there exists $\eta>0$ such that 
$$\frac{1}{\Card \C(T)} \Card \left\{ \g \in \C(T)\,  , \,  \frac{\hat{c}}{\ell(c)} \notin \mathcal{U} \right\} = O(e^{-\eta T}).$$
Moreover $\eta := \inf_{\nu \notin \mathcal{U}}  \{1 -h(\nu)\}$ where $h(\nu)$ is the entropy of $\varphi_t$ with respect to $\nu$.
\end{theorem}
The fact that $\eta>0$ is a consequence that $\hat{L}$ is the only measure of maximal entropy and  $h(\hat{L}) =1$.
 
This theorem has been used by Y. Herrera in his thesis \cite{herrera2013intersection} to estimate the self-intersection number of a random geodesic. His method inspired the proof of our main result.

\section{Isolation theorem}\label{sec :main result}
We are now ready to enter into the proof of the main isolation theorems. 

\begin{theorem}\label{th - isolation 1 surface}
Let $S_0$ be a fixed hyperbolic surface and $S_n$ be a sequence of hyperbolic surfaces. Then 
$\lim_{n\tv\infty}\delta(S_0,S_n)=1/2$ if and only if $\lim_{n\tv\infty} d_{Th}(S_0,S_n) = 0 $. 
\end{theorem}
Or in the Lorentzian language : 
\begin{TheoremNoCount}
Let $M_n$ be a sequence of GHMC $\AdS$ manifolds parametrize by $(\rho_0,\rho_n)$ then \\
$\lim_{n\tv \infty} \delta_{Lor}(M_n) =1 $ if and only if $\lim_{n\tv\infty} d_{Th}(S_0,S_n) = 0 $. 
\end{TheoremNoCount}

\begin{proof}
One way is just the continuity of critical exponent at $S_0$. Let's prove it briefly. If $d_{Th}(S_n,S_0) \tv 0 $, for all $\epsilon>0$, there is a $n_0$  such that for all $c\in \C$,  and all $n\geq n_0$, 
$$1-\epsilon < \frac{\ell_n(c)}{\ell_0(c)} <1 + \epsilon.$$
Hence $ \sum_{c \in \C} e^{-s(\ell_0(c) +\ell_n(c))}  >  \sum_{c \in \C} e^{-s(\ell_0(c) +(1+\epsilon )\ell_0(c))} $
This implies that for any $\epsilon>0$, and for $n\geq n_0$ we have
$$\delta(S_0,S_n)\geq \frac{1}{2+\epsilon}.$$
Therefore, $\lim \delta(S_0,S_n)\geq \frac{1}{2}$. Recalling that  $1/2\geq \delta(S_0,S_n)$ for any surfaces gives the result.

Let us show the converse. Suppose by contradiction that $d_{Th}(S_0,S_n)$ doesn't tend to 0. If $S_n$ stays in a compact subset of $\Teich(S)$, it admits a converging subsequence, that we still denote by $S_n$. Denote by  $S_\infty$ its limit which by hypothesis is different of $S_0$, then $\delta(S_0,S_\infty) < 1/2$ by rigidity Theorem  \ref{Rigidite}.
By continuity of critical exponent for the Thurston metric, $\lim \delta(S_0,S_n) =\delta(S_0,S_\infty)<1/2$ which is absurd, hence we can assume that $S_n$ leaves every compacts set of $\Teich(S)$. 

By earthquake's Theorem \ref{Geology}, there is a path from $S_0$ to $S_n$ in $\Teich(S)$ following an earthquake line. The first step then consists in proving that along every earthquake paths, the length  of "most"  curves on $S_0$ are increasing. This will imply that the  Poincaré's series over these curves has a decreasing  critical exponent. 
Consider the following function,
$$\begin{array}{cccc}
f :  & \MLam_1(S_0) \times \GC & \tv & \R \\
      &(\Lam,\nu)& \mapsto & \frac{i(\EQ{\Lam}{1}{m_0},\nu)}{i(L_0,\nu)}
\end{array}$$
 which is continuous, since earthquakes and intersection are continuous, Theorems \ref{continuité des tremblement de terre} and  \ref{th - def intersection  et l'intersection est continue}.
 
And where $m_0$ is the metric on $S_0$, $\MLam_1(S_0)$ is the set of laminations of $m_0$-length 1 and $L_0$ is the Liouville current associated to $m_0$. The set $\MLam_1(S_0)$ is compact hence $g(\nu) := \min_{\Lam} f(\Lam,\nu)$ is well defined and continuous.

The  compactness of $\MLam_1(S_0)$   implies also the existence of $\Lam_0 \in \MLam_1(S_0) $  such that  $g(L_0) =\min_{\Lam} f(\Lam,L_0) = \frac{i(\EQ{\Lam_0}{1}{m_0},L_0)}{i(L_0,L_0)} $. Remark that  $\EQ{\Lam_0}{1}{m_0} \neq m_0$ since $\Lam_0$ is not the trivial lamination, hence by Corollary \ref{l'intersction est plus grande que } it follows that $g(L_0)>1$. Let $\epsilon >0$ such that $g(L_0) =1+\epsilon$.
Recall  that for a geodesic current $\mu$ we denote by $\widehat{\mu}$ its corresponding $\varphi_t$-invariant probability obtained by the product of $\mu$ and the length along fibers. 
Let $\U$  be the neighbourhood of $\widehat{L_0}$ in $\Pants$  defined by $\U := \left\{ \hat{\mu} \in \Pants \, , \, |g(\mu) -g(L_0)|<\epsilon/2 \right\}$, and consider the set
$$\C_\U := \left\{ c \in \C \, , \, \frac{\hat{c}}{\ell_0(c)} \in \U\right\}.$$
By definition, if $c$ is in $\C_\U$,   it satisfies, $g(c)-g(L_0)>-\epsilon/2$, that is $\min \frac{i(\EQ{\Lam}{1}{m_0},c)}{i(m_0,c)}> 1+\epsilon/2$. Equivalently, we have for any $\Lam\in \MLam_1(S_0)$
$$ \frac{\ell(\EQ{\Lam}{1}{m_0} (c)) }{\ell_0(c)}>1+\epsilon/2.$$
By convexity of length along earthquake, Theorem \ref{Convexity along earthquake}, we have for all $\Lam\in \MLam_1(S_0)$,  all $t>1$ and all  $c\in \C_\U$, 
$$\ell(\EQ{\Lam}{t}{m_0} (c) ) \geq  (1+t \epsilon/2) \ell_0(c).$$
This last inequality is what we mathematically meant, by saying that the length of "most" curves are increasing. Indeed Kifer's result, Theorem \ref{Large deviation au niveau des groupes}, says that $\Card \C_\U^c \cap \C(k)$ has smaller exponential growth than $\Card \C(k)$. 

Let's look to the Poincaré series associated to $(S_0,S_n)$,
$$P_{0,n}(s)   := \sum_{c\in \C} e^{-s(\ell_0(c) +\ell_n(c))}.$$
 By Corollary \ref{Geology}, there exists $\Lam_n \in \MLam_1(S_0)$ and $t_n$ such that $\EQ{\Lam_n}{t_n}{m_0} =m_n$, hence
$$P_{0,n}(s)   = \sum_{c\in \C} e^{-s(\ell_0(c) +\ell(\EQ{\Lam_n}{t_n}{m_0} (c) ) )}.$$
Now we divide this sum into two parts, the curves which are in $\C_\U$ and the others. 

We claim that the critical exponent of the Poincaré series associated to the curves which are in $\C_\U$, tends to $0$. Indeed, 
$S_n$ goes out of every compacts of $\Teich(S)$, hence by continuity of $(t,\Lam)\tv \EQ{\Lam}{t}{m_0}$, the sequences $t_n$ must tends to infinity, in particular is greater than $1$. Therefore: 

$$\sum_{c\in \C_\U} e^{-s(\ell_0(c) +l(\EQ{\Lam_n}{t_n}{m_0} (c) ) )} <\sum_{c\in \C_\U} e^{-s(\ell_0(c) (1+\epsilon t_n/2)  )}$$

It implies that the series $\sum_{c\in \C_\U} e^{-s(\ell_0(c) +\ell(\EQ{\Lam_n}{t_n}{m_0} (c) ) )} $ has critical exponent strictly less than $\frac{1}{1+\epsilon t_n/2}$. The fact that $t_n\tv \infty$ finishes the proof of the claim. 


In the second step of the proof we get an upper bound on the critical exponent of the Poincaré sum over $\C_\U ^c$. This relies on Kifer's Theorem and the directional critical exponent since we want that the length of the geodesics on the second factor  to be almost proportional to the length on the first.

Let $\lambda_n$ be the slope for which the directional critical exponent $\delta (\lambda_n) $ between $S_0$ and $S_n$ is maximal. Let $u>0$ and  $A_n (u) := \left\{ c\in \C \, , \, \left| \frac{\ell_n(c)}{\ell_0(c)}- \lambda_n \right| < u \right\}.$ By Theorem \ref{rigidite de l exposant critique directionel} and Corollary \ref{l'exposant critique directionnel est maximum  quand il est egal a deltaM} for any $u>0$ the critical exponent of 
$\sum_{c\in A_n(u)} e^{-s(\ell_0(c) +\ell_n(c))}$
is  equal to the critical exponent of the whole Poincaré series $P_{0,n}$, that is to say $\delta(S_0,S_n)$. Hence $\delta(S_0,S_n)$ is equal to the maximum of the critical exponent of the two following series : 
\begin{itemize}
\item$\sum_{c\in \C_\U \cap A_n(u)} e^{-s(\ell_0(c) +\ell_n(c))}$
\item $\sum_{c\in \C_\U^c \cap A_n(u)} e^{-s(\ell_0(c) +\ell_n(c))}$
\end{itemize} 
We saw  in the previous claim that the critical exponent of the first one goes to $0$, in particular it is strictly less that $\delta(S_0,S_n)$ for $n$ sufficiently large, since we suppose that $\delta(S_0,S_n) \tv 1/2$. Hence $\delta(S_0,S_n)$ is equal to the critical exponent of $\sum_{c\in \C_\U^c \cap A_n(u)} e^{-s(\ell_0(c) +\ell_n(c))}$. The end of the proof consists to show that this exponent cannot tends to $1/2$.

 For $c \in A_n(u)$, $\ell_n(c)\geq \ell_0(c)(\lambda_n -u),$ we then have
 $$\sum_{c\in \C_\U^c \cap A_n(u)} e^{-s(\ell_0(c) +\ell_n(c))} \leq \sum_{c\in \C_\U^c \cap A_n(u)} e^{-s\ell_0(c)(1+\lambda_n -u)}.$$
 By Theorem \ref{Large deviation au niveau des groupes}, the complementary set of $\C_\U$ is "small" for the metric $m_0$, that is there exists $\eta>0$ and $M>0$  such that $\Card (\C_\U^c \cap \C(T))\leq M \Card \left( \C(T) \right) e^{-\eta T} \leq M'e^{(1-\eta)T}.$
Let $\C_\U^c(k) = \{c \in  \C_\U^c \text{ and } \, \ell_0(c)\in [k,k+1) \}$.
\begin{eqnarray*}
 \sum_{c\in \C_\U^c \cap A_n(u) } e^{-s\ell_0(c)(1+\lambda_n -u)} &\leq & 
\sum_{k \in \N} \, \sum_{c\in \C_\U^c(k) }e^{-s\ell_0(c)(1+\lambda_n -u)} \\
&\leq & 
\sum_{k \in \N} \, \sum_{c\in \C_\U^c (k)} e^{-s k(1+\lambda_n -u)}. \\
\end{eqnarray*}
And since $\C_\U^c (k) \subset \C_\U^c \cap \C(k)$,
\begin{eqnarray*}
\sum_{c\in \C_\U^c \cap A_n(u) } e^{-s\ell_0(c)(1+\lambda_n -u)}   &\leq & 
\sum_{k \in  \N} \sum_{c\in \C_\U^c \cap \C(k)} e^{-s k(1+\lambda_n -u)} \\
&\leq & 
\sum_{k \in \N} \Card\{c\in \C_\U^c \cap \C(k)\} e^{-s k(1+\lambda_n -u)} \\
&\leq & 
\sum_{k \in \N} M' e^{(1-\eta )k} e^{-s k(1+\lambda_n -u)} \\
&\leq & 
\sum_{k \in \N} M' e^{k (1-\eta -s(1+\lambda_n -u))}. \\
\end{eqnarray*} 
 This finally implies that $\sum_{c\in \C_\U^c \cap A_n(u)} e^{-s\ell_0(c)(1+\lambda_n -u)}$ has critical exponent less than $\frac{1-\eta}{1+\lambda_n -u}$. Combined to the fact that the critical exponent of $P_{0,n} $ is less or equal to this last one, and taking the limit in $u\tv 0$, we get :
 $$\delta(S_0,S_n) \leq \frac{1-\eta}{1+\lambda_n }.$$
 Suppose that $\delta(S_0,S_n) \tv 1/2$, then by Corollary \ref{Si delta tends vers 1 lambda tends vers 1} we deduce $\lambda_n \tv 1$. Taking the limit $n\tv \infty $, we get $1/2 \leq \frac{1-\eta}{2}$ which is absurd. This concludes the proof. 
\end{proof}

Next we look at what happens if both surfaces change. Recall  that $\Teich_\epsilon (S)$ is the thick part of Teichmüller space: surfaces for which no closed geodesic has length less than $\epsilon$. The mapping class group preserves the length spectrum, hence acts on the thick part of Teichmüller. Recall the Mumford compactness theorem. 
\begin{theorem}\cite{Mumford}
For any $\epsilon>0$, $\Teich_\epsilon(S)/MCG$  is compact. 
\end{theorem}

So if a surface stays in $\Teich_\epsilon(S)$, we can send it in a fixed compact set by the mapping class group. This remark with the previous Theorem allows us to show : 
\begin{corollaire}\label{cor - isolation 2 surfaces}
Let $S_n$ and $S_n'$ be two sequences of hyperbolic surfaces. Suppose that at least  one the sequences stays in $\Teich_\epsilon(S)$ for some $\epsilon$. Then $\lim_{n\tv\infty}\delta(S_n,S'_n)= 1/2$  if and only if $\lim_{n\tv\infty} d(S_n,S_n')= 0$. 
\end{corollaire}
Or in the Lorentzian language 
\begin{theorem}
Let $(M_n)$  be a sequence of GHMC manifolds, parametrized by $(S_n,S_n')$. If one of the sequences stays in the thick part of $\Teich(S)$ then
$$\lim_{n\tv \infty}\delta_{Lor} (M_n) =1 \quad \text{iff} \quad \lim_{n\tv \infty} d_{Th}(S_n,S_n') =0.$$
\end{theorem}

\begin{proof}
We are going to prove that if  $\lim_{n\tv\infty}\delta(S_n,S'_n)= 1/2$  then $\lim_{n\tv\infty} d(S_n,S_n')= 0$, the other implication is again a consequence of the continuity of the critical exponent for the Thurston metric. By hypothesis we can suppose that $S_n$ stays in $\Teich_\epsilon(S)$. Hence there exists a compact $K$ in $\Teich(S)$ and $D_n$ in the mapping class group, such that $D_n (S_n) \in K$ for every $n$. Hence we can suppose that $D_n(S_n)$ converges to $S_\infty\in K$. 
 

Let $u>0$ for every $n$ sufficiently large, 
$$1-u\leq \frac{\ell_{D_n(S_n)} (c)}{\ell_\infty(c)}\leq 1+ u$$
Hence  the critical exponents satisfies :
$$(1-u)\delta_{S_\infty, D_n(S_n')} \leq \delta _{D_n(S_n),D_n(S_n')} \leq (1+u)\delta_{S_\infty, D_n(S_n')}.$$
Now since the mapping class group doesn't change the length spectrum, the critical exponent of $(S_n,S_n')$ is equal to the critical exponent of $(D_n(S_n),D_n(S_n'))$. If  $\lim_{n\tv\infty}\delta(S_n,S'_n)= 1/2$  then  $\lim_{n\tv\infty} \delta (D_n(S_n),D_n(S_n'))= 1/2$ and by the previous inequalities it follows that $\lim_{n\tv\infty} \delta(S_\infty, D_n(S_n')) = 1/2$ since $u$ is arbitrary small. By Theorem  \ref{th - isolation 1 surface} this implies that $\lim_{n\tv\infty} d(D_n(S_n'), S_\infty)=0$, which finally implies that $\lim_{n\tv\infty}d(S_n,S_n') = 0$.
\end{proof}

\end{document}